\documentclass{amsart}
\usepackage{amsfonts}
\usepackage{amssymb}
\usepackage{times}
\usepackage{xcolor}
\usepackage{tkz-graph}
\usepackage{url}
\usepackage{float}
\usepackage{tasks}

\usepackage{cite}
\usepackage{hyperref}

\usepackage{amsfonts}
\usepackage{amsmath}
\usepackage{eurosym}
\usepackage{geometry}

\newcommand{\red}{\color{red}}

\newcommand{\black}{\color{black}}

\usepackage{caption,booktabs}

\theoremstyle{plain}
\newtheorem{theorem}{Theorem}[section]

\newtheorem{corollary} {Corollary}[section]
 
\newtheorem{definition} {Definition}[section]
\newtheorem{example} {Example}[section]
 
\newtheorem{lemma} {Lemma}[section]

\newtheorem{proposition} {Proposition}[section]
\newtheorem{remark} {Remark}[section]

\renewenvironment{proof}[1][Proof]{\noindent\textbf{#1.} }{\ \rule{0.5em}{0.5em}}

\begin{document}

\title{The algebraic and geometric classification of nilpotent Lie triple systems up to dimension four}

\thanks{The second and the fourth authors were  supported by the Centre for Mathematics of the University of Coimbra - UIDB/00324/2020, funded by the Portuguese Government through FCT/MCTES. Third author is supported by the PCI of the UCA `Teor\'\i a de Lie y Teor\'\i a de Espacios de Banach',  by the PAI with project number FQM298 and  by the project FEDER-UCA18-107643. The fourth author was supported by the Spanish Government through the Ministry of Universities grant `Margarita Salas', funded by the European Union - NextGenerationEU}

\author[H. Abdelwahab]{Hani Abdelwahab}
\address{Hani Abdelwahab. 
\newline \indent Mansoura University, Faculty of Science, Department of Mathematics (Egypt).}
\email{{\tt haniamar1985@gmail.com}}

\author[E. Barreiro]{Elisabete Barreiro}
\address{Elisabete~Barreiro. \newline \indent University of Coimbra, CMUC, Department of Mathematics, Apartado 3008,
EC Santa Cruz,
3001-501 Coimbra
(Portugal).}
\email{{\tt mefb@mat.uc.pt}}

\author[A.J. Calder\'on]{Antonio J. Calder\'on}
\address{Antonio J. Calder\'on. \newline \indent University of Cádiz, Department of Mathematics, Puerto Real (Espa\~na).}
\email{{\tt ajesus.calderon@uca.es}}

\author[Amir Fernández Ouaridi]{Amir Fernández Ouaridi}
\address{Amir Fernández Ouaridi. \newline \indent University of Cádiz, Department of Mathematics, Puerto Real (Espa\~na).
\newline \indent University of Coimbra, CMUC, Department of Mathematics, Apartado 3008,
EC Santa Cruz,
3001-501 Coimbra
(Portugal).}
\email{{\tt amir.fernandez.ouaridi@gmail.com}}


\thispagestyle{empty}

\begin{abstract}
In this paper we generalize the Skjelbred–Sund method, used to classify nilpotent Lie algebras, in order to classify triple systems with non-zero annihilator. We develop this method with the purpose of classifying nilpotent Lie triple systems, obtaining from it the algebraic classification of the nilpotent Lie triple systems up to dimension four. Additionally, we obtain  the geometric classification of the variety of nilpotent Lie triple systems up to dimension four.

\bigskip

{\it 2020MSC}: 17A30,
17A40,
17B30,
17D99,
14D06,
14L30.

{\it Keywords}: Nilpotent Lie triple systems, algebraic classification, geometric classification, annihilator extension, non-associative triple systems.
\end{abstract}

\maketitle

\section{{\protect\Large   }Introduction}

In recent years, numerous works have been published on the algebraic \cite{ack, contr11, cfk19, degr3, usefi1, degr2, degr1, demir, degs22, hac18, ikm19,   ikv18, kkk18} and geometric \cite{ale, maria, contr11, bb14, BC99, cfk19, gkk19, gkp, GRH, GRH2, ikv17, kppv, kpv, kv16,kv17, S90} classification of different varieties of low-dimensional algebras. Meanwhile, although there are a few works on the subject \cite{Bouetou, Neher}, the classification of triple systems is a matter still to be explored. In this work, we give the algebraic and geometric classification of nilpotent Lie triple systems, including the graph of primary degenerations and the inclusion graph of orbit closures of the variety. Lie triple systems were introduced by Jacobson \cite{Jacobson} as a generalization of the Lie algebras to the ternary case. These structures arise on the study of the curvature tensor of the space of continuous groups of transformations with affine connection and without torsion \cite{yamaguti}.

The algebraic classification of the nilpotent Lie triple systems is obtained by calculating the annihilator extensions of the same variety but with smaller dimension. To compute these extensions, we develop a generalization of the classic Skjelbred-Sund method, adapted to Lie triple systems. This method was first used for the classification of nilpotent Lie algebras \cite{ss78} and has been used in the last years for the classification of different varieties of algebras. For example,
 non-Lie central extensions of   $4$-dimensional Malcev algebras \cite{hac16},
  non-associative central extensions of   $3$-dimensional Jordan algebras \cite{ha17},
  anticommutative central extensions of $3$-dimensional anticommutative algebras \cite{cfk182},
  central extensions of $2$-dimensional algebras \cite{cfk18}
and some others were described.
There is also an algebraic classification of
nilpotent associative algebras up to dimension $4$ \cite{degr1},
nilpotent Novikov algebras up to dimension $4$ \cite{kkk18},
nilpotent restricted Lie algebras up to dimension $5$ \cite{usefi1},
nilpotent Jordan algebras up to dimension $5$ \cite{ha16},
nilpotent Lie algebras up to dimension $6$ \cite{degr3, degr2},
nilpotent Malcev algebras up to dimension $6$ \cite{hac18},
nilpotent binary Lie algebras up to dimension $6$ \cite{ack,acf},
nilpotent algebras up to dimension $3$ \cite{cfk18, degs22}
nilpotent commutative algebras up to dimension $4$ and
nilpotent anticommutative up to dimension $5$ \cite{degs22}.

The geometric classification is an interesting matter. The problem of determining the irreducible components of a variety of low-dimensional algebras (with respect to the Zariski topology) has been studied in various works. This allows identifying the so-called rigid algebras of the variety, which correspond to those algebras whose orbit by the action of the generalized linear group is an irreducible component. This subject was studied
for $2$-dimensional pre-Lie algebras in \cite{bb09},  
 $2$-dimensional terminal algebras in \cite{cfk19},
 $3$-dimensional Novikov algebras in \cite{bb14},  
 $3$-dimensional Jordan algebras in \cite{gkp},  
 $3$-dimensional Jordan superalgebras in \cite{maria},
 $3$-dimensional Leibniz and $3$-dimensional anticommutative algebras  in \cite{ikv18},
 $4$-dimensional Lie algebras in \cite{BC99},
 $4$-dimensional Zinbiel and  $4$-dimensional nilpotent Leibniz algebras in \cite{kppv},
  $3$-dimensional nilpotent algebras, $4$-dimensional nilpotent commutative algebras and $5$-dimensional nilpotent anticommutative algebras \cite{degs22},
 $6$-dimensional nilpotent Lie algebras in \cite{S90,GRH}, 
 $6$-dimensional nilpotent  Malcev algebras in \cite{kpv} 
and   $2$-dimensional algebras in \cite{kv16}.
To obtain the geometric classification of the variety of nilpotent Lie triple systems, we used a similar approach as the one used in these cited papers.

This work is organized as follows. In the first section, we will introduce the basic definitions related to Lie triple systems. In the second section, we will develop a method to classify the annihilator extensions of a Lie triple system and we will prove that every nilpotent Lie triple system is an annihilator extension of another lower-dimensional nilpotent Lie triple system. In the third section, we will apply this method and we will obtain the algebraic classification of nilpotent Lie triple systems up to dimension four. In the fourth section, we will study the geometric classification of the variety of nilpotent Lie triple systems of dimension three and four.

Throughout this paper, unless stated otherwise, $\mathbb{F}$ denotes an
arbitrary ground field of characteristic not two. All vector spaces are
assumed to be finite dimensional and defined over $\mathbb{F}$.

\begin{definition}[\protect\cite{Lister}] \rm
A vector space $\mathcal{T}$ together with a trilinear map
$ \left[\mbox{\textendash}, \mbox{\textendash}, \mbox{\textendash} \right]
:\mathcal{T\times T\times T}\longrightarrow \mathcal{T}$ is   a \emph{
Lie triple system}  if the following identities are satisfied:

\begin{enumerate}
\item[(A1)] $\left[ x,y,z\right] +\left[ y,x,z\right] =0,$

\item[(A2)] $\left[ x,y,z\right] +\left[ y,z,x\right] +\left[ z,x,y\right]
=0,$

\item[(A3)] $\left[ u,v,\left[ x,y,z\right] \right] =\left[ \left[ u,v,x%
\right] ,y,z\right] +\left[ x,\left[ u,v,y\right] ,z\right] +\left[ x,y,%
\left[ u,v,z\right] \right] ,$
\end{enumerate}
for all $x,y,z,u,v$ in $\mathcal{T}$.  The \emph{dimension}  of the triple system $\mathcal{T}$ is the dimension of $\mathcal{T}$ as a vector space.
\end{definition}

\begin{example} \rm
 Any Lie algebra  with product $\left[ x,y\right] $  can be viewed as a
Lie triple system with respect to the operation defined as $\left[ x,y,z\right]: =\left[ \left[ x,y%
\right] ,z\right] $. A Malcev algebra  with product $\left[ x,y\right] $  is
a\textbf{\ }Lie triple system with respect to the operation defined as $\left[ x,y,z\right] :=2%
\left[ \left[ x,y\right] ,z\right] -\left[ \left[ y,z\right] ,x\right] -%
\left[ \left[ z,x\right] ,y\right] $ (see \cite{Loss}). Every Jordan algebra
with product $x\circ y$  is a\textbf{\ }Lie triple system with respect to the operation
$\left[ x,y,z\right]: =\left( y\circ z\right) \circ x-y\circ \left( z\circ
x\right) $ (cf. \cite{Jacbson}).
\end{example}

Let $\mathcal{T}$ be a Lie triple system. A subspace $\mathcal{I}$ of $%
\mathcal{T}$ is  an \emph{ideal }if $\left[ \mathcal{I},\mathcal{T},%
\mathcal{T}\right] \subset \mathcal{I}$. If $\mathcal{I}$ is an ideal of  $\mathcal{T}$, then the quotient vector space $\overline{\mathcal{T}}\mathcal{=L}/\mathcal{I}$
together with the  trilinear map
\begin{equation*}
\left[ x+\mathcal{I},y+\mathcal{I},z+\mathcal{I}\right] :=\left[ x,y,z\right]
+\mathcal{I},
\end{equation*}
for all $x,y,z, \in \mathcal{T}$, is also a Lie triple system.

\begin{definition}[\protect\cite{Lister}] \rm
The \emph{annihilator} of a Lie triple system$\ \mathcal{T}$ is the ideal $\mathrm{Ann}({\mathcal{T}}) : =\left\{ x\in \mathcal{T}:\left[ x,\mathcal{T},\mathcal{T}
\right] =0\right\} $. Obviously, $\left[ x,\mathcal{T},\mathcal{T}\right] =0$
if and only if $\left[ \mathcal{T},x,\mathcal{T}\right] =\left[ \mathcal{T},\mathcal{T}
,x\right] =0$. A Lie triple system $\mathcal{T}$  is called \emph{abelian}  if $\mathcal{T}=\mathrm{Ann} ( \mathcal{T} ) $.
\end{definition}

Let  $\mathcal{T}_{1}$ and $\mathcal{T}
_{2}$ be two  Lie triple systems. A linear map $\phi :\mathcal{T}_{1}\longrightarrow \mathcal{T}_{2}$ is a
\emph{homomorphism} of Lie triple systems  if $\phi \left( \left[ x,y,z\right]_{\mathcal{T}_{1}} \right) =\left[ \phi \left(
x\right) ,\phi \left( y\right) ,\phi \left( z\right) \right]_{\mathcal{T}_{2}} $ for all $%
x,y,z $ in $\mathcal{T}_{1}$. If $\phi :\mathcal{T}_{1}\longrightarrow
\mathcal{T}_{2}$ is a homomorphism of Lie triple systems, then $\ker \phi $ is an ideal
of $\mathcal{T}_{1}$ and $\phi \left( \mathcal{T}_{1}\right) \cong \mathcal{T%
}_{1}/\ker \phi $.

Given an ideal $\mathcal{I}$ of a Lie triple system $\mathcal{T}$, we may introduce the
following sequence of subspaces: $$\mathcal{I}^{\left( 0\right) }:=\mathcal{I%
}, \quad \mathcal{I}^{\left( n+1\right) }:=\left[ \mathcal{I}^{\left( n\right) },%
\mathcal{T},\mathcal{I}\right] +\left[ \mathcal{I},\mathcal{T},\mathcal{I}
^{\left( n\right) }\right] ,$$ for $n\geq 0$.

\begin{definition}[see \protect\cite{Hopkins}]
\rm
An ideal  $\mathcal{I}$ of a Lie triple system $\mathcal{T}$ is $\mathcal{T}$-\emph{nilpotent} if $\mathcal{I}^{\left( m\right) }=0$, for
some  $m\geq 0$. A Lie triple system $\mathcal{T}$ is \emph{nilpotent} if itself is $\mathcal{T}$-  nilpotent.
\end{definition}

Note that if $\mathcal{I=T}$, then $\mathcal{T}^{\left( n+1\right) }=\left[
\mathcal{T}^{\left( n\right) },\mathcal{T},\mathcal{T}\right] $, for all $%
n\geq 0$. So if $\mathcal{T}$ is nilpotent and $\mathcal{T\neq }0$, then $%
\mathrm{Ann}\left( \mathcal{T}\right) \neq 0$. Moreover, $\mathcal{T}$ is
nilpotent if and only if $\mathcal{T}/\mathrm{Ann}\left( \mathcal{T}\right) $ is
nilpotent (see \cite{Hopkins}).

\section{The algebraic classification of nilpotent Lie triple systems}

An algebraic classification of the class of Lie triple systems of dimension $n$ is a classification, up to isomorphisms, of this class.
Throughout this work, unless stated otherwise,
$\mathcal{T}$ denotes a Lie triple system and ${\mathbb{V}}$   a vector space.

\begin{definition} \rm
Let $\mathcal{T}$ be a Lie triple system and let $\mathbb{V}$ be a vector space. The vector space $Z^{3}\left( \mathcal{T},{\mathbb{V}}\right) $ is defined as the set of all trilinear maps $\theta :\mathcal{T}\times \mathcal{T}
\times \mathcal{T}\rightarrow {\mathbb{V}}$ such that:

\begin{enumerate}
\item[(B1)] $\theta (x,y,z)+\theta (y,x,z)=0,$

\item[(B2)] $\theta (x,y,z)+\theta (y,z,x)+\theta (z,x,y)=0,$

\item[(B3)] $\theta (u,v,\left[ x,y,z\right] )+\theta (\left[ v,u,x\right]
,y,z)+\theta (x,\left[ v,u,y\right] ,z)+\theta (x,y,\left[ v,u,z\right] )=0,$
\end{enumerate}
for all $x,y,z,u,v$ in $\mathcal{T}$.
The elements of $Z^{3}\left( \mathcal{T},{\mathbb{V}}\right) $ are
called \emph{cocycles}.
\end{definition}

Let $\mathcal{T}$ be a Lie triple system and $\mathbb{V}$ be a vector space. Given a trilinear map $\theta :\mathcal{T}\times \mathcal{T}
\times \mathcal{T}\rightarrow {\mathbb{V}}$, consider the vector space  $\mathcal{T}_{\theta }:=\mathcal{T}\boldsymbol{\oplus } {\mathbb{V}}
$ endowed with the trilinear map $\left[\mbox{\textendash}, \mbox{\textendash}, \mbox{\textendash}\right]_{\mathcal{T}_{\theta }}: {\mathcal{T}_{\theta }}^3 \rightarrow \mathcal{T}_{\theta }$ given by  $\left[ x+u,y+v,z+w\right] _{%
\mathcal{T}_{\theta }}:=\left[ x,y,z\right] _{\mathcal{T}}+\theta (x,y,z)$
for all $x,y,z\in \mathcal{T}$ and $u,v,w\in {\mathbb{V}}$. The following
result shows this construction $\mathcal{T}_{\theta }$ is a Lie triple system if and only if $\theta \in Z^{3}\left( \mathcal{T},{\mathbb{V}}\right) $.

\begin{lemma}
\label{Lie triple system}Let $\mathcal{T}$ be a Lie triple system, $\mathbb{V}$ be a vector space
and $\theta :\mathcal{T}\times \mathcal{T}\times \mathcal{T}\longrightarrow {%
\mathbb{V}}$ be a trilinear map. The pair $(\mathcal{T}_{\theta }, {\left[\mbox{\textendash}, \mbox{\textendash}, \mbox{\textendash}\right]}_{\mathcal{T}_{\theta }})$  is a Lie triple system if and only if
$\theta \in Z^{3}\left( \mathcal{T},{\mathbb{V}}\right) $.
\end{lemma}

\begin{proof}
Since $\mathcal{T}$ is a Lie triple system, from $\left( \text{A1}\right) $-$%
\left( \text{A3}\right) $ we have  the following identities:%
\begin{equation*}
\left[ x,y,z\right] _{\mathcal{T}\mathbf{_{\theta }}}+\left[ y,x,z\right] _{%
\mathcal{T}\mathbf{_{\theta }}}=\theta (x,y,z)+\theta (y,x,z),
\end{equation*}
\begin{equation*}
\left[ x,y,z\right] _{\mathcal{T}\mathbf{_{\theta }}}+\left[ y,z,x\right] _{%
\mathcal{T}\mathbf{_{\theta }}}+\left[ z,x,y\right] _{\mathcal{T}\mathbf{%
_{\theta }}}=\theta (x,y,z)+\theta (y,z,x)+\theta (z,x,y),
\end{equation*}
\begin{equation*}
\left[ u,v,\left[ x,y,z\right] _{\mathcal{T}\mathbf{_{\theta }}}\right] _{%
\mathcal{T}\mathbf{_{\theta }}}+\left[ \left[ v,u,x\right] _{\mathcal{T}
\mathbf{_{\theta }}},y,z\right] _{\mathcal{T}\mathbf{_{\theta }}}+\left[ x,%
\left[ v,u,y\right] _{\mathcal{T}\mathbf{_{\theta }}},z\right] _{\mathcal{T}
\mathbf{_{\theta }}}+\left[ x,y,\left[ v,u,z\right] _{\mathcal{T}\mathbf{%
_{\theta }}}\right] _{\mathcal{T}\mathbf{_{\theta }}}=
\end{equation*}
\begin{equation*}
\theta (u,v,\left[ x,y,z\right] )+\theta (\left[ v,u,x\right] ,y,z)+\theta
(x,\left[ v,u,y\right] ,z)+\theta (x,y,\left[ v,u,z\right] ),
\end{equation*}
for any $x,y,z,u,v$ in $\mathcal{T}_{\theta }$. So, $\mathcal{T}_{\theta }$ is a Lie triple system if and only if $\theta \in Z^{3}\left(
\mathcal{T},{\mathbb{V}}\right) $.
\end{proof}

\noindent If $%
\mathrm{dim} {\mathbb{V}} =s$, the Lie triple system  $\mathcal{T}_{\theta }$ is called an $s$-%
\emph{dimensional annihilator extension} of $\mathcal{T}$ by $\mathbb{V}$.

\begin{definition}\rm
Let $\mathcal{T}$ be a Lie triple system and let $\mathbb V$ be a vector space. The \emph{\ radical} of a trilinear map $\theta \in Z^{3}\left( \mathcal{T},{\mathbb{V}}\right)$ is the set
\begin{equation*}
\mathrm{Rad}(\theta) = \left\{x\in \mathcal{T}: \theta(x, \mathcal{T}, \mathcal{T}) = 0 \right\}.
\end{equation*}
Clearly, $\theta(x, \mathcal{T}, \mathcal{T}) = 0 $  if and only if  $\theta(\mathcal{T}, x, \mathcal{T})= \theta(\mathcal{T}, \mathcal{T}, x)=0$.
\end{definition}

\begin{lemma}
\label{annihilator of extension}Let $\mathcal{T}$ be a Lie triple system, let $\mathbb V$ be a vector space and $\theta \in Z^{3}\left( \mathcal{T},%
{\mathbb{V}}\right) $.\ Then

\begin{equation*}
   \mathrm{Ann}\left( \mathcal{T}_{\theta }\right)
=\left( \mathrm{Rad}\left( \theta \right) \cap \mathrm{Ann}\left( \mathcal{T}
\right) \right) \boldsymbol{\oplus }{\mathbb{V}}.
\end{equation*}
\end{lemma}

\begin{proof}
Since ${\mathbb{V\subset }}\mathrm{Ann}\left( \mathcal{T}_{\theta }\right) $%
, we may write $\mathrm{Ann}\left( \mathcal{T}_{\theta }\right) ={\mathbb{W}}
\boldsymbol{\oplus }{\mathbb{V}}$ where ${\mathbb{W}}$ is a subspace of $%
\mathcal{T}$. For any $x$ $\in \mathcal{T}$, we have $\left[ x,\mathcal{T}
_{\theta },\mathcal{T}_{\theta }\right] _{\mathcal{T}_{\theta }}=\left[ x,%
\mathcal{T},\mathcal{T}\right] _{\mathcal{T}}+\theta \left( x,\mathcal{T},%
\mathcal{T}\right) $. Thus, $x\in {\mathbb{W}}$ if and only if $x\in \mathrm{Rad}\left(
\theta \right) \cap \mathrm{Ann}\left( \mathcal{T}\right) $. So ${\mathbb{W=}}
\mathrm{Rad}\left( \theta \right) \cap \mathrm{Ann}\left( \mathcal{T}
\right) $.
\end{proof}

\begin{lemma}
Let $\mathcal{T}$ be a Lie triple system, let $\mathbb V$ be a vector space and $\theta \in Z^{3}\left( \mathcal{T},%
{\mathbb{V}}\right) $.
In this condition,  $\mathcal{T}_{\theta }$ is nilpotent if and only if  $\mathcal{T}$ is nilpotent.
\end{lemma}

\begin{proof}
Since $ \left[ \mathcal{T}_{\theta
}^{\left( n\right) },x+u,y+v\right]_{\mathcal{T}_{\theta }} =%
\left[ \mathcal{T}^{\left( n \right) },x,y \right] +\theta %
( \mathcal{T}^{\left( n \right) },x,y  ) $, for all $x,y \in \mathcal{T}$ and $u,v \in {\mathbb{V}}$, we conclude that $%
\mathcal{T}_{\theta }$ is nilpotent if and only if $\mathcal{T}$ is nilpotent.
\end{proof}

The following result shows that every Lie triple system with a non-zero annihilator is an annihilator extension of a smaller-dimensional Lie triple system.

\begin{lemma}
Let $\mathcal{T}$ be a $n$-dimensional Lie triple system with $\mathrm{dim}(\mathrm{Ann}(\mathcal{T}))=m\neq0$. Then there exists, up to isomorphism, an unique $(n-m)$-dimensional Lie triple system $\mathcal{T}
^{\prime }$ and a trilinear map $\theta \in Z^{3}\left( \mathcal{T%
}^{\prime },\mathrm{Ann}\left( \mathcal{T}\right) \right) $ with $\mathrm{Rad}
\left( \theta \right) \cap \mathrm{Ann}\left( \mathcal{T}^{\prime }\right) =0$
such that $\mathcal{T}\cong \mathcal{T}_{\theta }^{\prime }$ and $\mathcal{T}
/\mathrm{Ann}\left( \mathcal{T}\right) \cong \mathcal{T}^{\prime }$.
\end{lemma}

\begin{proof}
Let $\mathcal{T}^{\prime }$ be a linear complement of $\mathrm{Ann}\left(
\mathcal{T}\right) $ in $\mathcal{T}$. Define a linear map $P:\mathcal{T}
\longrightarrow \mathcal{T}^{\prime }$ by $P(x+v):=x$ for $x\in \mathcal{T}
^{\prime }$ and $v\in \mathrm{Ann}\left( \mathcal{T}\right) $, and define a
trilinear map on $\mathcal{T}^{\prime }$ by $\left[ x,y,z\right] _{\mathcal{T%
}^{\prime }}:=P(\left[ x,y,z\right] _{\mathcal{T}})$ for $x,y,z\in \mathcal{T}
^{\prime }$. Then  we get
\begin{eqnarray*}
P([x,y,z]_{\mathcal{T}}) &=&P([x-P(x)+P(x),y-P(y)+P(y),z-P(z)+P(z)]_{%
\mathcal{T}}) \\
&=&P\left( [P(x),P(y),P(z)]_{\mathcal{T}}\right) \\
&=&[P(x),P(y),P(z)]_{\mathcal{T}^{\prime }},
\end{eqnarray*}
for any $x,y,z\in \mathcal{T} $, thus $P$ is a homomorphism of Lie triple systems. So $P(\mathcal{T})=\mathcal{T}^{\prime }$ is a
Lie triple system and $\mathcal{T}/\mathrm{Ann}\left( \mathcal{T}\right) \cong \mathcal{T}
^{\prime }$, which give us the uniqueness. Now, define the trilinear map $%
\theta :\mathcal{T}^{\prime }\times \mathcal{T}^{\prime }\times \mathcal{T}
^{\prime }\longrightarrow \mathrm{Ann}\left( \mathcal{T}\right) $ by $\theta
(x,y,z):=\left[ x,y,z\right] _{\mathcal{T}}-\left[ x,y,z\right] _{\mathcal{T}
^{\prime }}$. Consequently,  $\mathcal{T}_{\theta }^{\prime }$ is   $\mathcal{T}
$, therefore, $\theta \in Z^{3}\left( \mathcal{T}^{\prime },%
\mathrm{Ann}\left( \mathcal{T}\right) \right) $ by
Lemma \ref{Lie triple system}, as well   $\mathrm{Rad}\left(
\theta \right) \cap \mathrm{Ann}\left( \mathcal{T}^{\prime }\right) =0$ by Lemma \ref{annihilator of extension}.
\end{proof}

Let $\mathcal{T}$ be a Lie triple system and let $\mathbb{V}$ be a vector space. We denote as usually
the set of all linear maps from $\mathcal{T}$ to $\mathbb{V}$ by $\mathrm{Hom%
}\left( \mathcal{T},{\mathbb{V}}\right) $. For $f\in \mathrm{Hom}\left(
\mathcal{T},{\mathbb{V}}\right) $, consider the map $\delta f:\mathcal{T}\times
\mathcal{T}\times \mathcal{T}\longrightarrow {\mathbb{V}}$ defined by $\delta
f(x,y,z):=f(\left[ x,y,z\right] )$, for any $x, y, z \in \mathcal{T}$, then $\delta f\in Z^{3}\left(
\mathcal{T},{\mathbb{V}}\right) $. The following definition arises:

\begin{definition} \rm
Let $\mathcal{T}$ be a Lie triple system  and let $\mathbb V$ be a vector space. Define the vector subspace of $Z^{3}\left( \mathcal{T},{\mathbb{V}}
\right) $ by
\begin{equation*}
B^{3}\left( \mathcal{T},{\mathbb{V}}\right) =\left\{ \delta f:f\in \mathrm{Hom}\left( \mathcal{T},{\mathbb{V}}\right) \right\}  .
\end{equation*}
The elements of $B^{3}
\left( \mathcal{T},{\mathbb{V}}\right) $ are
called \emph{coboundaries}.
\end{definition}

\begin{lemma}
Let $\mathcal{T}$ be an $n$-dimensional Lie triple system and let $\left\{
e_{1},e_{2},\ldots ,e_{m}\right\} $ be a basis of $\mathcal{T}^{\left(
1\right) }=\left[\mathcal{T}, \mathcal{T}, \mathcal{T}\right]$. Then the set $\left\{\delta e_1^*, \delta e_2^*, \ldots, \delta e_m^*\right\}$, where $e_i^*(e_j) := \delta_{ij}$ and $\delta_{ij}$ is the Kronecker's symbol for $i,j \in \left\{
1,2,\ldots ,m\right\}$, is a basis of   $B^{3}\left( \mathcal{T},{\mathbb{F}}\right)
 $.
\end{lemma}

\begin{proof}
Let $\left\{ e_{1},\ldots ,e_{m},e_{m+1},\ldots ,e_{n}\right\} $ be a basis
of $\mathcal{T}$, by extending the basis of $\left[\mathcal{T}, \mathcal{T},\mathcal{T}\right]$. Then $\left\{ e_{1}^{\ast },e_{2}^{\ast },\ldots
,e_{n}^{\ast }\right\} $ is a basis of the dual space $\mathcal{T}^{\ast }=%
\mathrm{Hom}\left( \mathcal{T},{\mathbb{F}}\right) $. For any $\delta f\in
B^{3}\left( \mathcal{T},{\mathbb{F}}\right) $, with $%
f=\sum_{l=1}^{n}\alpha _{l}e_{l}^{\ast }$, we have

\begin{equation*}
    \delta f (e_i, e_j, e_k)= \sum_{l=1}^n \alpha_l e_l^*(\left[e_i, e_j, e_k\right] )= \sum_{l=1}^n \alpha_l e_l^*\Bigl(\sum_{p=1}^m \tau_{ijk}^p e_p\Bigr)=\sum_{l=1}^m \alpha_l e_l^*(\left[e_i, e_j, e_k\right] )= \sum_{l=1}^m \alpha_l \delta e_l^*(e_i, e_j, e_k).
\end{equation*}

\noindent Therefore  $\delta f=\sum_{l=1}^{m}\alpha _{l}\delta e_{l}^{\ast }$, proving
that $\delta e_{1}^{\ast },\delta e_{2}^{\ast },\ldots ,\delta e_{m}^{\ast }$
spans $B^{3}\left( \mathcal{T},{\mathbb{F}}\right) $. Moreover,
let $\alpha _{1},\ldots ,\alpha _{m}\in {\mathbb{F}}$ be such that $%
\sum_{l=1}^{m}\alpha _{l}\delta e_{l}^{\ast }=0$, then
\begin{equation*}
\Bigl(\sum_{l=1}^{m}\alpha _{l}e_{l}^{\ast }\Bigr)\left( \left[\mathcal{T}, \mathcal{T},\mathcal{T}\right]\right) =\Bigl(\sum_{l=1}^{m}\alpha _{l}\delta e_{l}^{\ast }\Bigr)(\mathcal{T},\mathcal{T},%
\mathcal{T})=0.
\end{equation*}
This implies that $\alpha _{1}=\alpha _{2}=\cdots =\alpha _{m}=0$ and
consequently $\delta e_{1}^{\ast },\delta e_{2}^{\ast },\ldots ,\delta
e_{m}^{\ast }$ are linearly independent.
\end{proof}

Let $\mathcal{T}$ be a Lie triple system and let $\mathbb{V}$ be a vector space. We
define $H^{3}\left( \mathcal{T},\mathbb{V}\right) $ as the
quotient space $Z^{3}\left( \mathcal{T},\mathbb{V}\right) \big/%
B^{3}\left( \mathcal{T},\mathbb{V}\right) $. The equivalence
class of $\theta \in Z^{3}\left( \mathcal{T},\mathbb{V}\right) $
will be denoted $\left[ \theta \right] \in H^{3}\left( \mathcal{T},%
\mathbb{V}\right) $.

\begin{lemma}
\label{equal}Let $\mathcal{T}$ be a Lie triple system and let $\mathbb{V}$ be a vector
space. Given two cocycles $\theta ,\vartheta \in Z^{3}\left(
\mathcal{T},{\mathbb{V}}\right) $ such that $\left[ \theta \right] =\left[
\vartheta \right] $, then $\mathrm{Ann}\left( \mathcal{T}_{\theta }\right) =%
\mathrm{Ann}\left( \mathcal{T}_{\vartheta }\right) $. Moreover, $\mathcal{T}
_{\theta }\cong \mathcal{T}_{\vartheta }$.
\end{lemma}

\begin{proof}
Suppose $\left[ \theta \right] =\left[ \vartheta \right] $, then $\vartheta
=\theta +\delta f$ for some $f\in \mathrm{Hom}\left( \mathcal{T},{\mathbb{V}}
\right) $. So for all $x,y,z\in \mathcal{T}$, we have 
$$\vartheta
(x,y,z)=\theta (x,y,z)+f(\left[ x,y,z\right] ).$$
Hence, $\theta (x,y,y)=\left[
x,y,z\right] =0$ if and only if $\vartheta (x,y,z)=\left[ x,y,z\right] =0$.
Therefore, $\mathrm{Rad}\left( \theta \right) \cap \mathrm{Ann}\left(
\mathcal{T}\right) =\mathrm{Rad}\left( \vartheta \right) \cap \mathrm{Ann}
\left( \mathcal{T}\right) $, so $\mathrm{Ann}\left( \mathcal{T}_{\theta
}\right) =\mathrm{Ann}\left( \mathcal{T}_{\vartheta }\right) $ by Lemma \ref%
{annihilator of extension}. Further,  we define a linear map $\varphi :\mathcal{%
L}_{\theta }\rightarrow \mathcal{T}_{\vartheta }$ by $\varphi (x+v):=x+f(x)+v$
for $x\in \mathcal{T}$ and $v\in {\mathbb{V}}$. Clearly $\varphi $ is
bijective. Moreover, if $x,y,z\in \mathcal{T}$ and $u,v,w\in {\mathbb{V}}$,
then we have
\begin{eqnarray*}
\varphi \Bigl( \left[ x+u,y+v,z+w\right] _{\mathcal{T}_{\theta
} }\Bigr)
&=&\varphi \Bigl( \left[ x,y,z\right] _{\mathcal{T}}+\theta \left(
x,y,z\right) \Bigr) \\
&=&\left[ x,y,z\right] _{\mathcal{T}}+f\left( \left[ x,y,z\right] _{\mathcal{%
T}}\right) +\theta \left( x,y,z\right) \\
&=&\left[ x,y,z\right] _{\mathcal{T}}+\delta f\left( x,y,z\right) +\theta
\left( x,y,z\right) \\
&=&\left[ x,y,z\right] _{\mathcal{T}}+\left( \delta f+\theta \right) \left(
x,y,z\right) \\
&=&\left[ x,y,z\right] _{\mathcal{T}}+\vartheta \left( x,y,z\right) \\
&=&\left[ x+f(x)+u,y+f(y)+v,z+f(z)+w\right] _{\mathcal{T}_{\vartheta }} \\
&=&\left[ \varphi (x+u),\varphi (y+v),\varphi (z+w)\right] _{\mathcal{T}
_{\vartheta }},
\end{eqnarray*}
hence, $\mathcal{T}_{\theta }$ and $\mathcal{T}_{\vartheta }$ are isomorphic.
\end{proof}

Let $\mathrm{Aut}\left( \mathcal{T}\right) $ be the automorphism group of
the Lie triple system $\mathcal{T}$ and ${\mathbb{V}}$ be a vector space. For an automorphism $\phi \in
\mathrm{Aut}\left( \mathcal{T}\right) $ and a cocycle $\theta \in Z^{3}\left( \mathcal{T},{\mathbb{V}}\right) $, define $\phi \theta :\mathcal{T}\times \mathcal{T}\times \mathcal{T}\longrightarrow {\mathbb{V}}$ by $\phi
\theta \left( x,y,z\right) :=\theta \left( \phi \left( x\right) ,\phi \left(
y\right) ,\phi \left( z\right) \right) $, for all $x,y,z\in \mathcal{T}$. Then $\phi \theta \in Z^{3}\left( \mathcal{T},{\mathbb{V}}\right) $. This is an action of the group  $\mathrm{Aut}\left(
\mathcal{T}\right) $  on $Z^{3}\left( \mathcal{T},{\mathbb{V}}
\right) $.

\begin{lemma}
Let $\mathcal{T}$ be a Lie triple system and let $\mathbb{V}$ be a vector space. For an automorphism $%
\phi \in \mathrm{Aut}\left( \mathcal{T}\right) $ and a cocycle $\theta \in Z^{3}\left( \mathcal{T},{\mathbb{V}}\right) $, the map $\phi \theta \in
B^{3}\left( \mathcal{T},{\mathbb{V}}\right) $ if and only if $\theta \in
B^{3}\left( \mathcal{T},{\mathbb{V}}\right) $.
\end{lemma}

\begin{proof}
Let
$\theta \in B^{3}\left( \mathcal{T},{\mathbb{V}}\right) $,
so $\theta =\delta f$ for some $f\in \mathrm{Hom}\left( \mathcal{T},{\mathbb{%
V}}\right) $. We have

\begin{equation*}
\begin{split}
    \phi \theta(x, y, z) & = \theta(\phi(x), \phi(y), \phi(z)) = \delta f(\phi(x), \phi(y), \phi(z))
     = f(\left[\phi(x), \phi(y), \phi(z)\right])\\ &=f\left( \phi \left( \left[ x,y,z\right] \right) \right)= \delta(f\circ \phi)(x, y, z),
\end{split}
\end{equation*}
for any $x,y,z\in \mathcal{T} $, hence $\phi \theta =\delta (f\circ \phi )\in B^{3}\left( \mathcal{L%
},{\mathbb{V}}\right) $. Conversely, Suppose $\phi \theta  \in B^{3}\left( \mathcal{T},{\mathbb{V}}\right) $, so $\phi \theta =\delta f$ for some $%
f\in \mathrm{Hom}\left( \mathcal{T},{\mathbb{V}}\right) $, then%
\begin{eqnarray*}
\theta \left( x,y,z\right) &=&\theta \left( \phi \left( \phi ^{-1}\left(
x\right) \right) ,\phi \left( \phi ^{-1}\left( y\right) \right) ,\phi \left(
\phi ^{-1}\left( z\right) \right) \right)  =\phi \theta \left( \phi ^{-1}\left( x\right) ,\phi ^{-1}\left( y\right)
,\phi ^{-1}\left( z\right) \right)\\
&=& \delta f (\phi^{-1}(x), \phi^{-1}(y), \phi^{-1}(z))
=f(\left[\phi^{-1}(x), \phi^{-1}(y), \phi^{-1}(z)\right] ) \\
&=&f\left( \phi ^{-1}\left( \left[ x,y,z\right] \right) \right) =\delta \left( f\circ \phi ^{-1}\right) \left( x,y,z\right),
\end{eqnarray*}
for any $x,y,z\in \mathcal{T} $, hence $\theta =\delta (f\circ \phi ^{-1})\in B^{3}\left( \mathcal{L%
},{\mathbb{V}}\right) $.
\end{proof}

Since $B^{3}\left( \mathcal{T},{\mathbb{V}}\right) $ is invariant under the action of $%
\mathrm{Aut}\left( \mathcal{T}\right) $,  we have an induced action of $\mathrm{Aut}\left( \mathcal{T}
\right) $   on $H^{3}\left( \mathcal{T},{\mathbb{V}}\right) $.
Our goal is to find all annihilator extensions of a  Lie triple system  $\mathcal{T}$ by a  vector space $\mathbb{V}$. In order to solve the isomorphism problem, we need to study the action of $\mathrm{Aut}\left( \mathcal{T}
\right) $   on $H^{3}\left( \mathcal{T},{\mathbb{V}}\right) $.

\begin{lemma}
Let $\mathcal{T}_{1},\mathcal{T}_{2}$ be two Lie triple systems and let ${\mathbb{V}}$ be a
vector space. If $\mathcal{T}_{1}$ and $\mathcal{T}_{2}$ are isomorphic, then it is verified the following equivalences:
\begin{enumerate}
\item[(1)] $Z^{3}\left( \mathcal{T}_{1},{\mathbb{V}}\right) \cong
Z^{3}\left( \mathcal{T}_{2},{\mathbb{V}}\right) $.

\item[(2)] $B^{3}\left( \mathcal{T}_{1},{\mathbb{V}}\right) \cong
B^{3}\left( \mathcal{T}_{2},{\mathbb{V}}\right) $.

\item[(3)] $H^{3}\left( \mathcal{T}_{1},{\mathbb{V}}\right) \cong
H^{3}\left( \mathcal{T}_{2},{\mathbb{V}}\right) $.
\end{enumerate}
\end{lemma}

\begin{proof}
$\left( 1\right) $ Consider an isomorphism $\phi :\mathcal{T}_{1}\longrightarrow \mathcal{T}_{2}$.
We can prove that $\theta \in Z^{3}\left( \mathcal{T}_{1},{%
\mathbb{V}}\right) $ if and only if $\phi ^{-1}\theta \in Z^{3}\left(
\mathcal{T}_{2},{\mathbb{V}}\right) $, where  $\phi ^{-1} \theta :\mathcal{T}_{2}\times \mathcal{T}_{2}\times \mathcal{T}_{2}\longrightarrow {\mathbb{V}}$ is defined by $\phi ^{-1}
\theta \left( x,y,z\right) :=\theta \left( \phi ^{-1} \left( x\right) ,\phi ^{-1} \left(
y\right) ,\phi ^{-1} \left( z\right) \right) $, for $x,y,z\in \mathcal{T}_{2}$. This defines an isomorphism $\tau :%
Z^{3}\left( \mathcal{T}_{1},{\mathbb{V}}\right) \longrightarrow
Z^{3}\left( \mathcal{T}_{2},{\mathbb{V}}\right) $ by $\tau \left(
\theta \right) :=\phi ^{-1}\theta $.\\
$\left( 2\right) $ Take $\theta \in B^{3}\left( \mathcal{T}_{1},{\mathbb{V}}\right) $, so $\theta =\delta f$
for some $f\in \mathrm{Hom}\left( \mathcal{T}_{1},{\mathbb{V}}\right) $. Then $\tau \left( \theta \right) =\tau \left( \delta f \right) =\phi ^{-1} \left( \delta f \right) =
\delta (f\circ \phi ^{-1})\in \mathcal{B}
^{3}\left( \mathcal{T}_{2},{\mathbb{V}}\right) $, since $f\circ \phi ^{-1}\in
\mathrm{Hom}\left( \mathcal{T}_{2},{\mathbb{V}}\right) $. On the other hand,
take $\vartheta \in B^{3}\left( \mathcal{T}_{2},{\mathbb{V}}\right) $, so $\vartheta =\delta f\in B^{3}\left( \mathcal{T}_{2},{\mathbb{V}}
\right) $ for some $f\in \mathrm{Hom}\left( \mathcal{T}_{2},{\mathbb{V}}
\right) $. Then $\tau \left( \theta \right) =\vartheta $, where $\theta
=\delta (f\circ \phi )\in B^{3}\left( \mathcal{T}_{1},{\mathbb{V}}
\right) $. So $\tau :B^{3}\left( \mathcal{T}_{1},{\mathbb{V}}
\right) \longrightarrow B^{3}\left( \mathcal{T}_{2},{\mathbb{V}}
\right) $ is an isomorphism.\\ $\left( 3\right) $ It follows from $\left(
1\right) $\ and $\left( 2\right) $.
\end{proof}

\begin{lemma}
\label{iso1}
Let $\mathcal{T}$ be a Lie triple system and let $\mathbb{V}$ be a vector space.
If $\theta \in Z^{3}\left( \mathcal{T},{\mathbb{V}}
\right) $ and $\phi \in \mathrm{Aut}\left( \mathcal{T}\right) $,  then $%
\mathcal{T}_{\theta }\cong \mathcal{T}_{\phi \theta }$ and $\mathrm{Rad}
\left( \theta \right) \cong \mathrm{Rad}\left( \phi \theta \right) $.
\end{lemma}

\begin{proof}
Consider the linear map $\varphi :\mathcal{T}_{\phi \theta }\longrightarrow
\mathcal{T}_{\theta }$ defined by $\varphi \left( x+v\right) :=\phi \left( x\right)
+v $, for $x\in \mathcal{T}$ and $v\in {\mathbb{V}}$. Then $\varphi $ is bijective. For any $x,y,z\in \mathcal{T}$
and $u,v,w\in $ ${\mathbb{V}}$, we have%
\begin{eqnarray*}
\varphi \Bigl(\left[ x+u,y+v,z+w\right] _{\mathcal{T}_{\phi \theta }}\Bigr) &=&\varphi
\Bigl( \left[ x,y,z\right] _{\mathcal{T}}+\phi \theta \left(
x,y,z\right) \Bigr) \\
&=&\phi( \left[ x,y,z\right] _{\mathcal{T}})+\phi \theta \left(
x,y,z\right)  \\
&=&\left[ \phi \left( x\right) ,\phi \left( y\right) ,\phi \left( z\right) %
\right] _{\mathcal{T}}+\theta \left( \phi \left( x\right) ,\phi \left(
y\right) ,\phi \left( z\right) \right) \\
&=&\left[ \varphi \left( x+u\right) ,\varphi \left( y+v\right) ,\varphi
\left( z+w\right) \right] _{\mathcal{T}_{\theta }},
\end{eqnarray*}
hence $\varphi $ is an isomorphism of Lie triple systems. With respect to the radical,  for $x\in \mathcal{T}$  we obtain $%
\phi \theta \left( x,\mathcal{T},\mathcal{T}\right) =\theta \left( \phi
\left( x\right) ,\mathcal{T},\mathcal{T}\right) $. So $x\in \mathrm{Rad}
\left( \phi \theta \right) $ if and only if $\phi \left( x\right) \in \mathrm{Rad}
\left( \theta \right) $, therefore $\varphi \mid_{\mathrm{Rad}\left( \phi \theta
\right)} :\mathrm{Rad}\left( \phi \theta
\right) \longrightarrow \mathrm{Rad}\left(  \theta \right) $ is
bijective.
\end{proof}

Denote as usually  the set of linear isomorphisms
from ${\mathbb{V}}$ to itself by $GL\left( {\mathbb{V}}\right) $. Given an automorphism $\psi \in GL\left( {\mathbb{V}}\right) $ and a cocycle  $\theta \in Z^{3}\left( \mathcal{T},{\mathbb{V}}\right) $,
define $\psi \theta \left( x,y,z\right): =\psi \left( \theta \left(
x,y,z\right) \right) $, for any $x,y,z\in \mathcal{T}$. Then $\psi \theta \in Z^{3}\left( \mathcal{T},{\mathbb{V}}\right) $. So  it is an action of the group  $GL\left( {\mathbb{V}}\right) $   on $%
Z^{3}\left( \mathcal{T},{\mathbb{V}}\right) $.

\begin{lemma}
Let $\mathcal{T}$ be a Lie triple system and let $\mathbb{V}$ be a vector space. For $%
\theta \in Z^{3}\left( \mathcal{T},{\mathbb{V}}\right) $ and $\psi
\in GL\left( {\mathbb{V}}\right) $, we have  $\psi \theta \in B^{3}\left( \mathcal{T},{\mathbb{V}}\right) $ if and only if $\theta \in B^{3}\left( \mathcal{T},{\mathbb{V}}\right) $.
\end{lemma}

\begin{proof}
Suppose that $\psi \theta \in B
^{3}\left( \mathcal{T},{\mathbb{V}}\right) $. Let $\psi \theta =\delta f$ for some $f\in \mathrm{Hom}\left( \mathcal{T},{%
\mathbb{V}}\right) $. Then
\begin{equation*}
\theta (x,y,z)=\psi ^{-1}\psi (\theta (x,y,z))=\psi ^{-1}(\delta
f(x,y,z))=(\psi ^{-1}\circ f)(\left[ x,y,z\right] )=\delta (\psi ^{-1}\circ
f)(x,y,z),
\end{equation*}
 for any $x,y,z\in \mathcal{T}$. Hence, $\theta =\delta (\psi ^{-1}\circ f)\in B^{3}\left( \mathcal{T},{\mathbb{V}}\right) $.  

Conversely, assume that $\theta \in B^{3}\left( \mathcal{T},{\mathbb{V}}\right) $ and let $\theta =\delta f$ for some $f\in
\mathrm{Hom}\left( \mathcal{T},{\mathbb{V}}\right) $. Then%
\begin{equation*}
\psi \theta \left( x,y,z\right) =\psi (\delta f)\left( x,y,z\right) =\left(
\psi \circ f\right) \left( \left[ x,y,z\right] \right) =\delta (\psi \circ
f)\left( x,y,z\right),
\end{equation*}
 for any $x,y,z\in \mathcal{T}$. So $\psi \theta =\delta (\psi \circ f)\in B^{3}\left( \mathcal{T},{%
\mathbb{V}}\right) $.
\end{proof}

Since $B^{3}\left( \mathcal{T},{\mathbb{V}}\right) $ is invariant under the action of $%
GL\left( {\mathbb{V}}\right) $  in the space $Z^{3}\left(
\mathcal{T},{\mathbb{V}}\right) $,  we have an induced action of $%
GL\left( {\mathbb{V}}\right) $   on $H^{3}\left( \mathcal{T},{\mathbb{V}}\right) $.

\begin{lemma}
\label{iso2}Let $\mathcal{T}$ be a Lie triple system and let $\mathbb{V}$ be a vector space. If  $\theta \in B^{3}\left( \mathcal{T},{\mathbb{V}}
\right) $ and $\psi \in GL({\mathbb{V}})$, then $\mathcal{T}_{\theta }\cong
\mathcal{T}_{\psi \theta }$ and $\mathrm{Rad}\left( \theta \right) =\mathrm{%
Rad}\left( \psi \theta \right) $.
\end{lemma}

\begin{proof}
Define a linear map $\varphi :\mathcal{T}_{\theta }\longrightarrow \mathcal{T%
}_{\psi \theta }$ by $\varphi (x+v):=x+\psi \left( v\right) $ for $x\in
\mathcal{T}$ and $v\in {\mathbb{V}}$. Then $\varphi $ is a bijective map.
Also, for any $x,y,z\in \mathcal{T}$ and $u,v,w\in {\mathbb{V}}$, we have%
\begin{eqnarray*}
\varphi \Bigl( \left[ x+u,y+v,z+w\right] _{\mathcal{T}_{\theta }}\Bigr)
=\varphi \Bigl( \left[ x,y,z\right] _{\mathcal{T}}+\theta \left(
x,y,z\right) \Bigr) =\left[ x,y,z\right] _{\mathcal{T}}+ (\psi \theta) \left( x,y,z\right) =\left[   x+u  ,  y+v   ,z+w  \right] _{%
\mathcal{T}_{\psi \theta }},
\end{eqnarray*}
hence, $\mathcal{T}_{\theta }\cong \mathcal{T}_{\psi \theta }$. With respect to the radical, since $\theta
\left( x,\mathcal{T},\mathcal{T}\right) =\psi ^{-1}\psi (\theta (x,\mathcal{T%
},\mathcal{T}))=\psi ^{-1} ((\psi \theta) (x,\mathcal{T%
},\mathcal{T}))$  we conclude that $x\in \mathrm{Rad}
\left( \psi \theta \right) $ if and only if $  x  \in \mathrm{Rad}
\left( \theta \right) $.
 
\end{proof}

\begin{lemma}
\label{iso3}Let $\mathcal{T}$ be a Lie triple system, let $\mathbb{V}$ be a vector space and let $\theta ,\vartheta \in Z^{3}\left( \mathcal{T},{%
\mathbb{V}}\right) $. If there exist a map $\phi \in Aut\left( \mathcal{T}
\right) $ and a map $\psi \in GL({\mathbb{V}})$ such that $\left[ \phi
\theta \right] =\left[ \psi \vartheta \right] $, then $\mathcal{T}_{\theta
}\cong \mathcal{T}_{\vartheta }$.
\end{lemma}

\begin{proof}
Let $\phi \in Aut\left( \mathcal{T}\right) $ and $\psi \in GL({\mathbb{V}})$
such that $\left[ \phi \theta \right] =\left[ \psi \vartheta \right] $.
Then by  Lemma \ref{equal} we have $\mathcal{T}_{\phi \theta }\cong \mathcal{T}
_{\psi \vartheta }$, by Lemma \ref{iso1}  we get $ 
\mathcal{T}_{\theta }\cong \mathcal{T}_{\phi \theta } $, and by Lemma \ref{iso2} we obtain $   \mathcal{T}_{\psi
\vartheta }\cong \mathcal{T}_{\vartheta }$.
\end{proof}

Now, using the two actions on $  Z^{3}\left( \mathcal{T},{%
\mathbb{V}}\right) $ that we have established before, we can state the following result.

\begin{lemma}
\label{iso4}Let $\mathcal{T}$ be a Lie triple system and let $\mathbb{V}$ be a vector space. Let $\theta ,\vartheta  $ be two cocycles in $  Z^{3}\left( \mathcal{T},{%
\mathbb{V}}\right) $ such that $\mathrm{Ann}\left( \mathcal{T}_{\theta
}\right) =\mathrm{Ann}\left( \mathcal{T}_{\vartheta }\right) ={\mathbb{V}}$.
In this conditions, $\mathcal{T}_{\theta }\cong \mathcal{T}_{\vartheta }$ if and only if
there exist an automorphism $\phi \in Aut\left( \mathcal{T}\right) $ and a linear isomorphism $\psi
\in GL({\mathbb{V}})$ such that $\left[ \phi \theta \right] =\left[ \psi
\vartheta \right] $.
\end{lemma}

\begin{proof}
If there exist a map $\phi \in Aut\left( \mathcal{T}
\right) $ and a map $\psi \in GL({\mathbb{V}})$ such that $\left[ \phi
\theta \right] =\left[ \psi \vartheta \right] $ then  $\mathcal{T}_{\theta
}\cong \mathcal{T}_{\vartheta }$ by Lemma \ref{iso3}.
Now let us show the converse. Consider $\theta ,\vartheta \in Z^{3}\left( \mathcal{T},{\mathbb{V}}
\right) $ such that $\mathrm{Ann}\left( \mathcal{T}_{\theta }\right) =%
\mathrm{Ann}\left( \mathcal{T}_{\vartheta }\right) ={\mathbb{V}}$. Suppose
that $\mathcal{T}_{\theta }\cong \mathcal{T}_{\vartheta }$,  there
exists an isomorphism $\Phi :$ $\mathcal{T}_{\theta }\longrightarrow
\mathcal{T}_{\vartheta }$. Since $\Phi \left( {\mathbb{V}}\right) =\Phi
\left( \mathrm{Ann}\left( \mathcal{T}_{\theta }\right) \right) =\mathrm{Ann}
\left( \mathcal{T}_{\vartheta }\right) ={\mathbb{V}}$,   we define $  \psi := \Phi|_{\mathbb V}  \in GL\left( {\mathbb{V}}\right) $. Let $e_{1},e_{2},\ldots,e_{n}$
be a basis of $\mathcal{T}$, and let $\Phi \left( e_{i}\right)
=e_{i}^{\prime }+v_{i}$, where $e_{i}^{\prime }\in \mathcal{T}$ and $v_{i}$ $%
\in V$, for every $i \in \{ 1,\ldots ,n \}$. Then $\Phi $ induces an automorphism $\phi  :\mathcal{T}
\longrightarrow \mathcal{T}$ defined by $\phi  \left( e_{i}\right)
:=e_{i}^{\prime }$,\ and a linear map $\varphi :\mathcal{T}\longrightarrow {%
\mathbb{V}}$ defined by $\varphi \left( e_{i}\right): =v_{i}$. So we can
realize $\Phi $ as a matrix of the form%
\begin{equation*}
\Phi =%
\begin{pmatrix}
\phi   & 0 \\
\varphi & \psi%
\end{pmatrix}  ,
\end{equation*} 
where $\phi  \in Aut\left( \mathcal{T}\right)$, $\psi =\Phi |_{{\mathbb{V}}}\in
GL\left( {\mathbb{V}}\right)$ and  $\varphi \in \mathrm{Hom}\left(
\mathcal{T},{\mathbb{V}}\right) $.
Furthermore, for any $x,y,z\in \mathcal{T}$ we have%
\begin{equation*}
\Phi \left( \left[ x,y,z\right] _{\mathcal{T}_{\theta }}\right) =\Phi \left( %
\left[ x,y,z\right] _{\mathcal{T}}+\theta \left( x,y,z\right) \right) =\phi
 (\left[ x,y,z\right]_{\mathcal{T}} )+\varphi (\left[ x,y,z\right]_{\mathcal{T}} )+\psi (\theta
\left( x,y,z\right) ),
\end{equation*}
and%
\begin{eqnarray*}
\left[ \Phi \left( x\right) ,\Phi \left( y\right) ,\Phi \left( z\right) %
\right] _{\mathcal{T}_{\vartheta }} &=&\left[ \phi  \left( x\right)
+\varphi \left( x\right) ,\phi  \left( y\right) +\varphi \left( y\right)
,\phi  \left( z\right) +\varphi \left( z\right) \right] _{\mathcal{T}
_{\vartheta }} \\
&=&\left[ \phi  \left( x\right) ,\phi  \left( y\right) ,\phi
 \left( z\right) \right] _{\mathcal{T}}+\vartheta \left( \phi  \left(
x\right) ,\phi  \left( y\right) ,\phi  \left( z\right) \right) \\
&=&\phi  (\left[ x,y,z\right] _{\mathcal{T}})+\vartheta \left( \phi
 \left( x\right) ,\phi  \left( y\right) ,\phi  \left( z\right)
\right) .
\end{eqnarray*}
Since $\Phi $ is an isomorphism, it follows that%
\begin{equation}
\vartheta \left( \phi  \left( x\right) ,\phi  \left( y\right) ,\phi
 \left( z\right) \right) =\varphi \left( \left[ x,y,z\right] _{\mathcal{T}
}\right) +\psi \left( \theta \left( x,y,z\right) \right) ,  \label{auto}
\end{equation}
 for all   $x,y,z\in  \mathcal{T}$. Hence we have $\phi  \vartheta =\delta \varphi +\psi \theta $ and
 $\left[ \phi  \vartheta \right] =\left[ \psi \theta \right] $.
\end{proof}

\noindent In case of $\theta =\vartheta $, we obtain from Condition $\left( \ref{auto}\right) $
the following description of $Aut(\mathcal{T}_{\theta })$.

\begin{corollary}
\label{Aut}Let $\mathcal{T}$ be a Lie triple system, let $\mathbb{V}$ be a vector space  and $\theta \in Z%
^{3}\left( \mathcal{T},{\mathbb{V}}\right) $ such that $\mathrm{Rad}\left(
\theta \right) \cap \mathrm{Ann}\left( \mathcal{T}\right) =0$. Then the
automorphism group $Aut(\mathcal{T}_{\theta })$ consists of all linear maps
of the matrix form
\begin{equation*}
\Phi =%
\begin{pmatrix}
\phi  & 0 \\
\varphi & \psi%
\end{pmatrix} , %
\end{equation*}
where $\phi  \in Aut\left( \mathcal{T}\right)$, $\psi  \in
GL\left( {\mathbb{V}}\right)$ and  $\varphi \in \mathrm{Hom}\left(
\mathcal{T},{\mathbb{V}}\right) $,
such that $\theta \left( \phi  \left( x\right) ,\phi  \left( y\right)
,\phi \left( z\right) \right) =\varphi \left( \left[ x,y,z\right] _{%
\mathcal{T}}\right) +\psi \left( \theta \left( x,y,z\right) \right) \ $ for
any $x,y,z\in \mathcal{T}$.
\end{corollary}

\begin{definition} \rm
Let $\mathcal{T}$ be a Lie triple system. If   $x\in \mathcal{T}$ satisfies $%
x\notin \mathcal{T}^{\left( 1\right) } = \left[ \mathcal{T},\mathcal{T},\mathcal{T} \right]$ and $x\in \mathrm{Ann}\left( \mathcal{%
T}\right) $, we call $\mathbb{F}x$ an \emph{annihilator component }of $%
\mathcal{T}$.
\end{definition}

\begin{remark} \rm
 If $\phi :\mathcal{T}_{1}\rightarrow \mathcal{T}_{2}$ is an
isomorphism of Lie triple systems, then $\mathcal{T}_{1}$ has a annihilator
component if and only if so has $\mathcal{T}_{2}$. In fact, consider the vectors $x\in \mathcal{T}_{1}$ and $y \in \mathcal{T}_{2}$ such that $\phi(x) = y$, it follows that  $\mathbb{F}x$ is an annihilator component of $\mathcal{T}_{1}$ if and only if $\mathbb{F}y$ is an annihilator component of $\mathcal{T}_{2}$.
\end{remark}

Let $\mathbb V$ be an $s$-dimensional vector space and let $e_{1},e_{2},\ldots ,e_{s}$ be a fixed basis of ${\mathbb{V}}$. Then $%
\theta \in Z^{3}\left( \mathcal{T},{\mathbb{V}}\right) $ can be uniquely written as $\theta \left( x,y,z\right) = \sum_{i=1}^s
 \theta _{i}\left( x,y,z\right) e_{i},$ where $\theta
_{i}\in Z^{3}\left( \mathcal{T},{\mathbb{F}}\right) $. Moreover,
$\mathrm{Rad}\left( \theta \right) =\cap_{i=1}^s 
$%
$\mathrm{Rad}\left( \theta _{i}\right) $.

\begin{lemma}
\label{cent.comp}Let $\theta \left( x,y,z\right) =\sum_{i=1}^s\theta _{i}\left( x,y,z\right) e_{i}\in Z^{3}\left(
\mathcal{T},{\mathbb{V}}\right) $ and $\mathrm{Rad}\left( \theta \right)
\cap \mathrm{Ann}\left( \mathcal{T}\right) =0$. Then $\mathcal{T}_{\theta }$
has an annihilator component if and only if $\left[ \theta _{1}\right] ,\left[ \theta _{2}
\right] ,\ldots ,\left[ \theta _{s}\right] $ are linearly dependent in $%
H^{3}\left( \mathcal{T},{\mathbb{F}}\right) $.
\end{lemma}

\begin{proof}
Suppose $\mathcal{T}_{\theta }$ has an 
annihilator component 
and fix an annihilator component $\mathbb{F}
v_{1}\subseteq {\mathbb{V}}$. Enlarge the set $\left\{ v_{1}\right\} $ to
a set $\left\{ v_{1},v_{2},\ldots,v_{s}\right\} $ to form a basis of ${%
\mathbb{V}}$. Then there exists an invertible matrix   $\big(a_{ij}\big)$ change of basis such
that  $e_{i}=\sum_{j=1}^s%
a_{ij}v_{j}$, for any $i=1,\ldots ,s$. So $$\theta \left( x,y,z\right) =\underset{j=1}{\overset{s}{%
\sum }}\Bigl( \overset{s}{\underset{i=1}{\sum }}a_{ij}\theta _{i}\left(
x,y,z\right) \Bigr) v_{j}.$$ Since $v_{1}\notin \mathcal{T}_{\theta
}^{\left( 1\right) }$, it follows that  $\sum_{i=1}^s a_{i1}\theta
_{i}\left( x,y,z\right) =0$ for all $x,y,z\in \mathcal{T}$. Then $\sum_{i=1}^s a_{i1}\theta _{i}=0$, and therefore $\sum_{i=1}^s a_{i1}\left[ \theta _{i}\right] =0$. Since $\det \big(%
a_{ij}\big)\neq 0$, then  $\left[ \theta _{1}\right] ,\left[ \theta _{2}\right]
,\ldots ,\left[ \theta _{s}\right] $ are linearly dependent in $H%
^{3}\left( \mathcal{T},{\mathbb{F}}\right) $.

Conversely, suppose that $%
\left[ \theta _{1}\right] ,\left[ \theta _{2}\right] ,\ldots ,\left[ \theta
_{s}\right] $ are linearly dependent and $\left[ \theta _{s}\right] =
\sum_{i=1}^{s-1} \alpha _{i}\left[ \theta _{i}\right] $.
Now, define a new cocycle $\vartheta \left( x,y,z\right) =\sum_{i=1}^s%
\vartheta _{i}\left( x,y,z\right)e_{i} $ by setting $\vartheta _{i}=\theta _{i}$
for $i=1,\ldots ,s-1$ and $\vartheta _{s}=\sum_{i=1}^{s-1}
 \alpha _{i}\theta _{i}$. Then $\left[ \theta \right] =\left[ \vartheta %
\right] $ and by Lemma \ref{equal} we get $\mathcal{T}_{\theta }\cong \mathcal{T}_{\vartheta }$.
Further, we have $\vartheta \left( x,y,z\right) =\sum_{i=1}^{s-1}\theta _{i}\left( x,y,z\right) \left( e_{i}+\alpha _{i}e_{s}\right) $%
. For $i=1,\ldots ,s-1$, set $w_{i}:=e_{i}+\alpha _{i}e_{s}$. Then $\vartheta
\left( x,y,z\right) =\sum_{i=1}^{s-1}\theta _{i}\left(
x,y,z\right) w_{i}$. Hence, $\mathcal{T}_{\vartheta }^{\left( 1\right) }=  \left[ \mathcal{T}_{\vartheta },\mathcal{T}_{\vartheta },\mathcal{T}_{\vartheta } \right]  \subset
\mathcal{T}\oplus \left\langle w_{1},w_{2},\ldots ,w_{s-1}\right\rangle $,
so that $\mathcal{T}_{\vartheta }$, and therefore also $\mathcal{T}_{\theta
} $, has a annihilator component.
\end{proof}

\noindent The statement in Lemma \ref{iso4} can be rephrased as follows.

\begin{lemma}
\label{one2one1}
Given two cocycles $\theta(x, y, z) = \sum_{i=1}^s \theta_i(x, y, z)e_i$ and $\vartheta(x, y, z) = \sum_{i=1}^s \vartheta_i(x, y, z)e_i$ in $Z^{3}\left( \mathcal{T},{%
\mathbb{V}}\right) $. Suppose that $\mathcal{T}_{\theta }$ has no annihilator component and $\mathrm{Rad}\left( \theta \right) \cap \mathrm{Ann}\left(
\mathcal{T}\right) =\mathrm{Rad}\left( \vartheta \right) \cap \mathrm{Ann}
\left( \mathcal{T}\right) =0$. In these conditions,  the Lie triple systems $\mathcal{T}_{\theta }$ and $\mathcal{T}
_{\vartheta }$ are isomorphic if and only if there exists an automorphism $\phi\in Aut\left( \mathcal{T}\right)$ such that the set $\left\{\left[\phi\vartheta_i\right]: i = 1, \ldots, s\right\}$ spans the same subspace of $H^{3}\left( \mathcal{T},\mathbb{F}\right) $ as the set $\left\{\left[\theta_i\right]: i = 1, \ldots, s\right\}$.

\end{lemma}

\begin{proof}
Suppose first that $\mathcal{T}_{\theta }\cong \mathcal{T}_{\vartheta }$.
Then, by Lemma \ref{iso4}, there exist a map $\phi \in Aut\left( \mathcal{T}
\right) $ and a map $\psi \in GL\left( {\mathbb{V}}\right) $ such that $%
\left[ \phi \vartheta \right] =\left[ \psi \theta \right] $. Let $\psi
\left( e_{i}\right) = \sum_{j=1}^s a_{ij}e_{j}$. Then
\begin{equation}
\left( \phi \vartheta -\psi \theta \right) \left( x,y,z\right) =\underset{j=1%
}{\overset{s}{\sum }}\Bigl( \phi \vartheta _{j}-\underset{i=1}{\overset{s}{%
\sum }}a_{ij}\theta _{i}\Bigr) \left( x,y,z\right) e_{j}.  \label{eq2}
\end{equation}
Since $\phi \vartheta _{j}-\sum_{i=1}^s a_{ij}\theta
_{i}\in B^{3}\left( \mathcal{T},\mathbb{F}\right) $, then  $\left[
\phi \vartheta _{j}\right] =\sum_{i=1}^s a_{ij}\left[
\theta _{i}\right] $. Hence, $\left\{ \left[ \phi \vartheta _{i}\right]
:i=1,\ldots ,s\right\} $ spans the same subspace of $H^{3}\left(
\mathcal{T},\mathbb{F}\right) $ as the set $\left\{ \left[ \theta _{i}
\right] :i=1,\ldots ,s\right\} $.

Conversely, suppose $\left\{\left[\phi\vartheta_i\right]: i = 1, \ldots, s\right\}$ spans the same vector space as  $\left\{\left[\theta_i\right]: i = 1, \ldots, s\right\}$. Then there exists an
invertible matrix $\big(a_{ij}\big)$ such that $\left[ \phi \vartheta _{j}
\right] =\sum_{i=1}^s a_{ij}\left[ \theta _{i}\right] $%
. Define a linear map $\psi :{\mathbb{V}}\longrightarrow {\mathbb{V}}$ by $%
\psi \left( e_{i}\right) :=\sum_{j=1}^s a_{ij}e_{j}$.
Then $\psi \theta \left( x,y,z\right) =\sum_{i=1}^s%
\sum_{j=1}^s a_{ij}\theta _{i}\left( x,y,z\right) e_{j}$
and we have Equation \eqref{eq2}. Hence, $\left[ \phi \vartheta \right] =\left[ \psi
\theta \right] $,  and  therefore  $\mathcal{T}_{\theta }\cong \mathcal{T}_{\vartheta }$
by Lemma \ref{iso4}.
\end{proof}

Let $G_{s}\left( H^{3}\left( \mathcal{T},\mathbb{F}\right) \right)
$ be the Grassmannian of subspaces of dimension $s$ in $H^{3}\left( \mathcal{T},\mathbb{F}\right) $. The automorphism group $%
Aut\left( \mathcal{T}\right) $ acts on $G_{s}\left( H^{3}\left(
\mathcal{T},\mathbb{F}\right) \right) \ $as follows: if $\phi \in Aut\left(
\mathcal{T}\right) $ and vector space  $\mathbb{W}=\left\langle \left[ \theta _{1}\right] ,%
\left[ \theta _{2}\right] ,\ldots,\left[ \theta _{s}\right] \right\rangle \in
G_{s}\left( H^{3}\left( \mathcal{T},\mathbb{F}\right) \right) $,
then $\phi \mathbb{W}:=\left\langle \left[ \phi \theta _{1}\right] ,\left[
\phi \theta _{2}\right] ,\ldots,\left[ \phi \theta _{s}\right] \right\rangle
\in G_{s}\left( H^{3}\left( \mathcal{T},\mathbb{F}\right) \right) $%
. We denote the orbit of $\mathbb{W}\in G_{s}\left( H^{3}\left(
\mathcal{T},\mathbb{F}\right) \right) $ under the action of $Aut\left(
\mathcal{T}\right) $ by $\mathcal{O}\left( \mathbb{W}\right) $.

\begin{lemma}
Given two vector spaces ${\mathbb W}_1=\langle\left[\theta_1\right], \left[\theta_2\right], \ldots, \left[\theta_s \right]\rangle$ and ${\mathbb W}_2=\langle\left[\vartheta_1\right], \left[\vartheta_2\right], \ldots, \left[\vartheta_s \right]\rangle$ in $G_{s}\left( H
^{3}\left( \mathcal{T},\mathbb{F}\right) \right) $. If ${\mathbb W}_1 = {\mathbb W}_2$, then $\cap_{i=1}^s\mathrm{Rad}\left(
\theta _{i}\right) \cap \mathrm{Ann}\left( \mathcal{T}\right) =\cap_{i=1}^s\mathrm{Rad}\left( \vartheta _{i}\right) \cap \mathrm{Ann}
\left( \mathcal{T}\right) $.
\end{lemma}

\begin{proof}
Assume  $\mathbb{W}_{1}=\mathbb{W}_{2}$. Then there exists an invertible matrix $%
\big(a_{ij}\big)$ such that $\left[ \theta _{i}\right] =\sum_{j=1}^s a_{ij}\left[ \vartheta _{j}\right] $. Therefore, we have $\theta
_{i}=\sum_{j=1}^s a_{ij}\vartheta _{j}+\delta f_{i}$
for some $f_{i}\in \mathrm{Hom}\left( \mathcal{T},\mathbb{F}\right) $. Then, $$%
\theta _{i}\left( x,y,z\right) =\underset{j=1}{\overset{s%
}{\sum }}a_{ij}\vartheta _{j}\left( x,y,z\right)
+f_{i}\left( \left[ x,y,z\right] _{\mathcal{T}}\right) .$$%
 Hence, $\theta _{1}\left( x,y,z\right) =\cdots =\theta
_{s}\left( x,y,z\right) =\left[ x,y,z%
\right]_{\mathcal{T}} =0$ if and only if $\vartheta _{1}\left( x,y,z\right)
=\cdots =\vartheta _{s}\left( x,y,z\right) =\left[ x,%
y,z\right]_{\mathcal{T}} =0$. Therefore $\cap_{i=1}^s\mathrm{Rad}\left(
\theta _{i}\right) \cap \mathrm{Ann}\left( \mathcal{T}\right) =\cap_{i=1}^s\mathrm{Rad}\left( \vartheta _{i}\right) \cap \mathrm{Ann}
\left( \mathcal{T}\right) $.
\end{proof}

From this result, we can define the set%
\begin{equation*}
\mathcal{T}_{s}\left( \mathcal{T}\right) :=\left\{ \mathbb{W}=\left\langle %
\left[ \theta _{1}\right] ,\left[ \theta _{2}\right] ,\ldots,\left[ \theta _{s}
\right] \right\rangle \in G_{s}\left( H^{3}\left( \mathcal{T},%
\mathbb{F}\right) \right) :\underset{i=1}{\overset{s}{\cap }}\mathrm{Rad}
\left( \theta _{i}\right) \cap \mathrm{Ann}\left( \mathcal{T}\right)
=0\right\} .
\end{equation*}

\begin{lemma}
The set $\mathcal{T}_{s}\left( \mathcal{T}\right) $ is stable under the
action of $Aut\left( \mathcal{T}\right) $.
\end{lemma}

\begin{proof}
Let $\phi \in Aut\left( \mathcal{T}\right) $ and ${\mathbb{W}}=\left\langle %
\left[ \theta _{1}\right] ,\left[ \theta _{2}\right] ,\ldots,\left[ \theta _{s}
\right] \right\rangle \in G_{s}\left( H^{3}\left( \mathcal{T},%
\mathbb{F}\right) \right) $. Then, $x\in \cap_{i=1}^s
\mathrm{Rad}\left( \phi \theta _{i}\right) \cap \mathrm{Ann}\left( \mathcal{T}
\right) $ if and only if $\phi\left(x\right)\in \cap_{i=1}^s \mathrm{Rad}\left( \theta
_{i}\right) \cap \mathrm{Ann}\left( \mathcal{T}\right) $. So $\cap_{i=1}^s \mathrm{Rad}\left( \phi \theta _{i}\right) \cap \mathrm{Ann}\left( \mathcal{T}\right) =0$ if and only if $\cap_{i=1}^s %
\mathrm{Rad}\left( \theta _{i}\right) \cap \mathrm{Ann}\left( \mathcal{T}
\right) =0$. Consequently, $\phi {\mathbb{W}}\in \mathcal{T}_{s}\left( \mathcal{T}
\right) $ if and only if ${\mathbb{W}}\in \mathcal{T}_{s}\left( \mathcal{T}\right) $.
\end{proof}

Let $\mathcal{T}$ be a Lie triple system, let $\mathbb V$ be an $s$-dimensional vector space and let $\left\{e_1, e_2, \ldots, e_s\right\}$ be a basis of $\mathbb V$. Consider the set $\mathcal{E}
\left( \mathcal{T},{\mathbb{V}}\right) $ of all Lie triple systems without annihilator components, which are $s$-dimensional   annihilator
extensions of $\mathcal{T}$ by ${\mathbb{V}}$ and have $s$-dimensional
annihilator. Then%
\begin{equation*}
\mathcal{E}\left( \mathcal{T},{\mathbb{V}}\right) =\left\{ \mathcal{T}
_{\theta }:\theta \left( x,y,z\right) =\underset{i=1}{\overset{s}{\sum }}
\theta _{i}\left( x,y,z\right) e_{i}\mbox{
and }\left\langle \left[ \theta _{1}\right] ,\left[ \theta _{2}\right] ,\ldots,%
\left[ \theta _{s}\right] \right\rangle \in \mathcal{T}_{s}\left( {\mathcal{T}
}\right) \right\}  .
\end{equation*}
Given $\mathcal{T}_{\theta }\in \mathcal{E}\left( \mathcal{T},{\mathbb{V}}
\right) $, we denote by $\left[ \mathcal{T}_{\theta }\right] $  the isomorphism
class of $\mathcal{T}_{\theta }$.
Using this new notation, we can rephrase Lemma \ref{one2one1} as follows.
\begin{lemma}
Let $\mathcal{T}_{\theta },\mathcal{T}_{\vartheta }\in \mathcal{E}\left(
\mathcal{T},{\mathbb{V}}\right) $. Suppose that $\theta \left( x,y,z\right) =\sum_{i=1}^s \theta _{i}\left( x,y,z\right) e_{i}$ and $%
\vartheta \left( x,y,z\right) =\sum_{i=1}^s \vartheta
_{i}\left( x,y,z\right) e_{i}$. In this conditions, $\left[ \mathcal{T}_{\theta }\right] =%
\left[ \mathcal{T}_{\vartheta }\right] $ if and only if \emph{$\mathcal{O}$}$%
\left\langle \left[ \theta _{1}\right] ,\left[ \theta _{2}\right] ,\ldots,\left[
\theta _{s}\right] \right\rangle =$\emph{$\mathcal{O}$}$\left\langle \left[
\vartheta _{1}\right] ,\left[ \vartheta _{2}\right] ,\ldots,\left[ \vartheta
_{s}\right] \right\rangle $.
\end{lemma}

\begin{proof}
Consider two Lie triple systems $\mathcal{T}_{\theta },\mathcal{T}_{\vartheta }\in \mathcal{E}\left(
\mathcal{T},{\mathbb{V}}\right) $. By Lemma \ref{one2one1}, $\mathcal{T%
}_{\theta }\cong \mathcal{T}_{\vartheta }$ if and only if there exists an automorphism $\phi \in
Aut\left( \mathcal{T}\right) $ such that $\left\langle \left[ \phi \vartheta
_{1}\right] ,\left[ \phi \vartheta _{2}\right] ,\ldots,\left[ \phi \vartheta
_{s}\right] \right\rangle =\left\langle \left[ \theta _{1}\right] ,\left[
\theta _{2}\right] ,\ldots,\left[ \theta _{s}\right] \right\rangle $. Hence, $%
\left[ \mathcal{T}_{\theta }\right] =\left[ \mathcal{T}_{\vartheta }\right] $
if and only if $\mathcal{O}\left\langle \left[ \theta _{1}\right] ,\left[ \theta _{2}
\right] ,\ldots,\left[ \theta _{s}\right] \right\rangle =\mathcal{O}
\left\langle \left[ \vartheta _{1}\right] ,\left[ \vartheta _{2}\right] ,\ldots,%
\left[ \vartheta _{s}\right] \right\rangle $.
\end{proof}

Thus, each orbit of $Aut\left( \mathcal{T}\right) $ on $\mathcal{T}
_{s}\left( \mathcal{T}\right) $\ corresponds uniquely to an isomorphism
class of $\mathcal{E}\left( \mathcal{T},{\mathbb{V}}\right) $, and vice-versa. This correspondence is defined by%
\begin{equation*}
\mathcal{O}\left\langle \left[ \theta _{1}\right] ,\left[ \theta _{2}\right]
,\ldots,\left[ \theta _{s}\right] \right\rangle \in \left\{ \mathcal{O}\left(
\mathbb{W}\right) :\mathbb{W}\in \mathcal{T}_{s}\left( \mathcal{T}\right)
\right\} \longleftrightarrow \left[ \mathcal{T}_{\theta }\right] \in \left\{ %
\left[ \mathcal{T}_{\vartheta }\right] :\mathcal{T}_{\vartheta }\in \mathcal{%
E}\left( \mathcal{T},{\mathbb{V}}\right) \right\} \text{,}
\end{equation*}
where $ \theta \left( x,y,z\right) =\sum_{i=1}^s \theta
_{i}\left( x,y,z\right) e_{i}$. We call $\mathcal{T}_{\theta }$ the Lie triple system
corresponding to the representative $\left\langle \left[ \theta _{1}\right] ,%
\left[ \theta _{2}\right] ,\ldots,\left[ \theta _{s}\right] \right\rangle $. Finally, we have the following correspondence theorem.

\begin{theorem}
There exists a one-to-one correspondence between the set of $Aut\left(
\mathcal{T}\right) $-orbits on $\mathcal{T}_{s}\left( \mathcal{T}\right) $
and the set of isomorphism classes of $\mathcal{E}\left( \mathcal{T},{%
\mathbb{V}}\right) $.
\end{theorem}

By this theorem, we can construct all Lie triple systems of dimension $n$
with $s$-dimensional annihilator, given those algebras of dimension $n-s$,
following this steps:

\begin{enumerate}
\item For a Lie triple system $\mathcal{T}$ of dimension $n-s$, determine $%
H^{3}\left( \mathcal{T},\mathbb{F}\right) $, $\mathrm{Ann}\left(
\mathcal{T}\right) $ and $Aut\left( \mathcal{T}\right) $.

\item Determine the set of $Aut\left( \mathcal{T}\right) $-orbits on $%
\mathcal{T}_{s}\left( \mathcal{T}\right) $.

\item For each orbit, construct the Lie triple system corresponding to a
representative of it.
\end{enumerate}

\section{The algebraic classification of nilpotent Lie triple systems up to dimension four}

Let us introduce the following notations. 
 Let $\mathcal{T}$ be a Lie triple system with a basis $\left\{
e_{1},e_{2},\ldots ,e_{n}\right\} $. 
Let us now define a trilinear form $\Delta _{i,j,k}:\mathcal{T}\times
\mathcal{T}\times \mathcal{T}\longrightarrow {\mathbb{F}}$ by%
\begin{equation*}
\Delta _{i,j,k}(e_{l},e_{m},e_{t})=\left\{
\begin{tabular}{ll}
$1$ & if $\left( i,j,k\right) =\left( l,m,t\right) $, \\
$-1$ & if $\left( j,i,k\right) =\left( l,m,t\right) $, \\
$0$ & otherwise.%
\end{tabular}
\right.
\end{equation*}
Then $Z^{3}\left( \mathcal{T},{\mathbb{F}}\right) $ is a subspace
of $\left\langle \Delta _{i,j,k}:1\leq i<j\leq n,1\leq k\leq n\right\rangle $%
. In particular, we have $Z^{3}\left( \mathcal{T},{\mathbb{F}}\right) =0$ if $%
\mathrm{dim}\,\mathcal{T}=1$ and $\mathrm{dim}\,Z^{3}\left( \mathcal{T},{\mathbb{F}
}\right) \leq 2$ if $\mathrm{dim}\,\mathcal{T}=2$. Further, if $\mathrm{dim}\,\mathcal{T}\geq 3$%
, then $\mathrm{dim}\,Z^{3}\left( \mathcal{T},{\mathbb{F}}\right) \leq 2%
\dbinom{n+1}{3}$.

A cocycle $%
\theta \in Z^{3}\left( \mathcal{T},{\mathbb{F}}\right) $ is uniquely determined by its values $\theta
(e_{i},e_{j},e_{k})=\alpha _{i,j,k}$ and it may be represented by a block
matrix $\mathcal{C}=%
\begin{pmatrix}
\mathcal{C}_{1} & \mathcal{C}_{2} & \cdots & \mathcal{C}_{n}
\end{pmatrix}
$ where for  $t \in \left\{
1,2,\ldots ,n\right\} $ the $\mathcal{C}_{t}$  is the matrix representing the skew symmetric
bilinear form $ 
\mathcal{T}\times \mathcal{T}\longrightarrow {\mathbb{F}}$ defined by $ \left( x,y\right) \longmapsto \theta \left( x,y,e_{t}\right)$
(i.e., if $\mathcal{C}_{t}=\big(c_{ij}
\big)$ then $c_{ij}=\alpha _{i,j,t}$). We call $\mathcal{C}$ the \emph{%
matrix form }of\emph{\ }$\theta $.

Let $\phi =\big(a_{ij}\big)\in \mathrm{Aut}\left( \mathcal{T}\right) $, $%
\theta \in Z^{3}\left( \mathcal{T},{\mathbb{F}}\right) $, $%
\mathcal{C}=%
\begin{pmatrix}
\mathcal{C}_{1} & \mathcal{C}_{2} & \cdots & \mathcal{C}_{n}
\end{pmatrix}
$ be the matrix form of $\theta $ and $\mathcal{C}^{\prime }=%
\begin{pmatrix}
\mathcal{C}_{1}^{\prime } & \mathcal{C}_{2}^{\prime } & \cdots & \mathcal{C}
_{n}^{\prime }
\end{pmatrix}
$ be the matrix form of $\phi \theta $. Then
\begin{eqnarray*}
  \phi \theta \left( x,y,e_{k}\right)
=\theta \left( \phi \left( x\right) ,\phi \left( y\right) ,\phi \left(
e_{k}\right) \right) =\theta \Bigl( \phi \left( x\right) ,\phi \left( y\right) ,\underset{i=1}{%
\overset{n}{\sum }}a_{ik}e_{i}\Bigr) =\underset{i=1}{\overset{n}{\sum }}a_{ik}\theta \left( \phi \left(
x\right) ,\phi \left( y\right) ,e_{i}\right),
\end{eqnarray*}
thus, if $\mathcal{B}=\underset{i=1}{\overset{n}{\sum }}a_{ik}\mathcal{C}_{i}$%
, therefore $\mathcal{C}_{k}^{\prime }=\phi ^{t}\mathcal{B}\phi $.

\begin{example}
Consider the    abelian $2$%
-dimensional Lie triple system $\mathcal{T}:=\left\langle e_{1},e_{2}\right\rangle $. Then $Z^{3}\left( \mathcal{T},{\mathbb{F}}
\right) =\left\langle \Delta _{1,2,1},\Delta _{1,2,2}\right\rangle $ and $%
\mathrm{Aut}\left( \mathcal{T}\right) =\{\phi =\left( a_{ij}\right) \in {%
\mathcal{M}}_{2\times 2}({\mathbb{F}}):\det \phi \neq 0\}$. Let $\theta
=\alpha _{1,2,1}\Delta _{1,2,1}+\alpha _{1,2,2}\Delta _{1,2,2}$. Then the
matrix form of $\theta $ is $\mathcal{C}=%
\begin{pmatrix}
\mathcal{C}_{1} & \mathcal{C}_{2}
\end{pmatrix}
$ where $\mathcal{C}_{1}=%
\begin{pmatrix}
0 & \alpha _{1,2,1} \\
-\alpha _{1,2,1} & 0%
\end{pmatrix}
,\mathcal{C}_{2}=%
\begin{pmatrix}
0 & \alpha _{1,2,2} \\
-\alpha _{1,2,2} & 0%
\end{pmatrix}
$. If $\mathcal{C}^{\prime }=%
\begin{pmatrix}
\mathcal{C}_{1}^{\prime } & \mathcal{C}_{2}^{\prime }
\end{pmatrix}
$ is the matrix form of $\phi \theta $, then%
\begin{eqnarray*}
\mathcal{C}_{1}^{\prime } &=&\phi ^{t}\left( a_{11}
\begin{pmatrix}
0 & \alpha _{1,2,1} \\
-\alpha _{1,2,1} & 0%
\end{pmatrix}
+a_{21}
\begin{pmatrix}
0 & \alpha _{1,2,2} \\
-\alpha _{1,2,2} & 0%
\end{pmatrix}
\right) \phi , \\
\mathcal{C}_{2}^{\prime } &=&\phi ^{t}\left( a_{12}
\begin{pmatrix}
0 & \alpha _{1,2,1} \\
-\alpha _{1,2,1} & 0%
\end{pmatrix}
+a_{22}
\begin{pmatrix}
0 & \alpha _{1,2,2} \\
-\alpha _{1,2,2} & 0%
\end{pmatrix}
\right) \phi .
\end{eqnarray*}
\end{example}

\begin{proposition}
Let $\mathcal{T}$ be a Lie triple system and $\mathrm{dim}\,\mathcal{T}\leq 2$. If $\mathcal{T}$
is nilpotent, then $\mathcal{T}$ is abelian. 
\end{proposition}

\begin{proof}
Since $\mathcal{T}$ is nilpotent, $\mathrm{Ann}\left( \mathcal{T}\right) \neq
0 $. So if $\mathrm{dim}\,\mathcal{T}=1$, then $\mathcal{T}=\mathrm{Ann}\left( \mathcal{T}
\right) $. If $\mathrm{dim}\,\mathcal{T}=2$, then $\mathcal{T}$ is a annihilator
extension of $1$-dimensional Lie triple system $\mathcal{T}^{\prime }$. But $Z^{3}\left( \mathcal{T}^{\prime },{\mathbb{F}}\right) =0$. So, any annihilator
extension of $\mathcal{T}^{\prime }$ is trivial. Thus, $\mathcal{T}$ is
abelian.
\end{proof}

We denote by $\mathcal{T}_{1,1}$ the abelian Lie triple system of dimension $1$  and  by $\mathcal{T}_{2,1}$ the
abelian Lie triple system of dimension $2$.

\begin{remark}
\label{[1 0]}Let $X=%
\begin{pmatrix}
\alpha & \beta%
\end{pmatrix}
\in {\mathcal{M}}_{1\times 2}({\mathbb{F}})$ and $X\neq 0$. Then there
exists an invertible matrix $A\in {\mathcal{M}}_{2\times 2}({\mathbb{F}})$
such that $XA=%
\begin{pmatrix}
1 & 0%
\end{pmatrix}
$. To see this, suppose first that $\alpha \neq 0$. Then $%
\begin{pmatrix}
\alpha & \beta%
\end{pmatrix}
\begin{pmatrix}
\alpha ^{-1} & -\beta \\
0 & \alpha%
\end{pmatrix}
=%
\begin{pmatrix}
1 & 0%
\end{pmatrix}
$. Assume now that $\alpha =0$. Then $%
\begin{pmatrix}
0 & \beta%
\end{pmatrix}
\begin{pmatrix}
0 & 1 \\
\beta ^{-1} & 0%
\end{pmatrix}
=\allowbreak
\begin{pmatrix}
1 & 0%
\end{pmatrix}
$.
\end{remark}

\begin{proposition}
Let $\mathcal{T}$ be a nilpotent Lie triple system of dimension $3$. Then $\mathcal{T}$
is isomorphic to one of the following:

\begin{itemize}
\item $\mathcal{T}_{3,1}=\mathcal{T}_{2,1}\oplus \mathcal{T}_{1,1}$ (abelian
Lie triple system).

\item $\mathcal{T}_{3,2}:\left[ e_{1},e_{2},e_{1}\right] =e_{3}$.
\end{itemize}
\end{proposition}

\begin{proof}
Suppose that $\mathcal{T}$ is not abelian. Then $\mathcal{T}$ is a $1$%
-dimensional annihilator extension of $\mathcal{T}_{2,1}$. Since, we have $Z^{3}\left( \mathcal{T}_{2,1},{\mathbb{F}}\right) =\left\langle \Delta
_{1,2,1},\Delta _{1,2,2}\right\rangle $ and $B^{3}\left( \mathcal{T%
}_{2,1},{\mathbb{F}}\right) =0$, then $Z^{3}\left( \mathcal{T}_{2,1},%
{\mathbb{F}}\right) =H^{3}\left( \mathcal{T}_{2,1},{\mathbb{F}}
\right) $. Furthermore, we have $\mathrm{Aut}\left( \mathcal{T}
_{2,1}\right) = GL({\mathcal{T}
_{2,1}})$.

Now, choose an arbitrary subspace $\mathbb{W}\in \mathcal{T}_{1}\left( \mathcal{T}
_{2,1}\right) $. Then $\mathbb{W}$ is spanned by $\theta =\alpha
_{1,2,1}\Delta _{1,2,1}+\alpha _{1,2,2}\Delta _{1,2,2}$ such that $\left(
\alpha _{1,2,1},\alpha _{1,2,2}\right) \neq \left( 0,0\right) $. Let $\phi =%
\big(a_{ij}\big)\in $ $\mathrm{Aut}\left( \mathcal{T}_{2,1}\right) $. Write $%
\phi \theta =\beta _{1,2,1}\Delta _{1,2,1}+\beta _{1,2,2}\Delta _{1,2,2}$.
Then%
\begin{eqnarray*}
\beta _{1,2,1} &=&\left( a_{11}a_{22}-a_{12}a_{21}\right) \left(
a_{11}\alpha _{1,2,1}+a_{21}\alpha _{1,2,2}\right) , \\
\beta _{1,2,2} &=&\left( a_{11}a_{22}-a_{12}a_{21}\right) \left(
a_{12}\alpha _{1,2,1}+a_{22}\alpha _{1,2,2}\right) .
\end{eqnarray*}
Set $\alpha _{1,2,1}^{\prime }=\det \left( \phi \right) \alpha _{1,2,1}$ and
$\alpha _{1,2,2}^{\prime }=\det \left( \phi \right) \alpha _{1,2,2}$. Then $%
\beta _{1,2,1}=a_{11}\alpha _{1,2,1}^{\prime }+a_{21}\alpha _{1,2,2}^{\prime
},\beta _{1,2,2}=a_{12}\alpha _{1,2,1}^{\prime }+a_{22}\alpha
_{1,2,2}^{\prime }$ and $\left( \alpha _{1,2,1}^{\prime },\alpha
_{1,2,2}^{\prime }\right) \neq \left( 0,0\right) $. By Remark \ref{[1 0]},
there exists a $\phi \in $ $\mathrm{Aut}\left( \mathcal{T}_{2,1}\right) $
such that $\phi \theta =\Delta _{1,2,1}$. So we get the Lie triple system $\mathcal{T}
_{3,2}$.
\end{proof}

\begin{theorem} \label{classification4}
Let $\mathcal{T}$ be a nilpotent Lie triple system of dimension $4$ over an
algebraically closed field $\mathbb{F}$ of characteristic $\neq 2,3$. Then $%
\mathcal{T}$ is isomorphic to one of the following:

\begin{itemize}
\item $\mathcal{T}_{4,1}=\mathcal{T}_{3,1}\oplus \mathcal{T}_{1,1}$ (abelian
Lie triple system).

\item $\mathcal{T}_{4,2}=\mathcal{T}_{3,2}\oplus \mathcal{T}_{1,1}:\left[
e_{1},e_{2},e_{1}\right] =e_{3}.$

\item $\mathcal{T}_{4,3}:\left[ e_{1},e_{2},e_{1}\right] =e_{3},\left[
e_{1},e_{2},e_{2}\right] =e_{4}.$

\item $\mathcal{T}_{4,4}:\left[ e_{2},e_{3},e_{2}\right] =e_{4},\left[
e_{3},e_{1},e_{3}\right] =e_{4}.$

\item $\mathcal{T}_{4,5}:\left[ e_{2},e_{3},e_{1}\right] =e_{4},\left[
e_{3},e_{1},e_{2}\right] =e_{4},\left[ e_{2},e_{1},e_{3}\right] =2e_{4},%
\left[ e_{2},e_{3},e_{2}\right] =e_{4}.$

\item $\mathcal{T}_{4,6}^{\lambda }
:\left[ e_{1},e_{2},e_{3}\right] =-\left( \lambda +1\right) e_{4},\left[
e_{2},e_{3},e_{1}\right] =\lambda e_{4},\left[ e_{3},e_{1},e_{2}\right]
=e_{4}.$

\item $\mathcal{T}_{4,7}:\left[ e_{1},e_{2},e_{1}\right] =e_{3},\left[
e_{1},e_{2},e_{3}\right] =e_{4},\left[ e_{1},e_{3},e_{2}\right] =e_{4}.$

\item $\mathcal{T}_{4,8}:\left[ e_{1},e_{2},e_{1}\right] =e_{3},\left[
e_{1},e_{3},e_{1}\right] =e_{4},\left[ e_{1},e_{2},e_{2}\right] =e_{4}.$

\item $\mathcal{T}_{4,9}:\left[ e_{1},e_{2},e_{1}\right] =e_{3},\left[
e_{1},e_{3},e_{1}\right] =e_{4}.$
\end{itemize}

Among these Lie triple systems, there are precisely the following isomorphisms:

\begin{itemize}
\item $\mathcal{T}_{4,6}^{0} \cong \mathcal{T}_{4,6}^{-1}$.
\item $\mathcal{T}_{4,6}^{\alpha }\left( \alpha ^{2}+\alpha \neq 0\right)
\cong \mathcal{T}_{4,6}^{\beta }\left( \beta ^{2}+\beta \neq 0\right) $ if and only if $%
\allowbreak \xi \left( \alpha \right) =\xi \left( \beta \right) $ where%
\begin{equation*}
\xi \left( \lambda \right) =\frac{\left( \lambda ^{2}+\lambda +1\right) ^{3}
}{\lambda ^{2}\left( \lambda +1\right) ^{2}}:\lambda ^{2}+\lambda \neq 0.
\end{equation*}
\end{itemize}
\end{theorem}

\begin{proof}
Suppose that $\mathcal{T}$ has no annihilator components. On the one hand, if $\mathrm{dim}\,\mathrm{Ann}
\left( \mathcal{T}\right) =2$, then $\mathcal{T}$ is a $2$-dimensional
annihilator extension of $\mathcal{T}_{2,1}$. Since $H^{3}\left(
\mathcal{T}_{2,1},{\mathbb{F}}\right) =\left\langle \Delta _{1,2,1},\Delta
_{1,2,2}\right\rangle $, $\mathcal{T}_{1}\left( \mathcal{T}\right) =\{%
\mathbb{W=}\left\langle \Delta _{1,2,1},\Delta _{1,2,2}\right\rangle \}$. So
we get the Lie triple system $\mathcal{T}_{4,3}$. On the other hand, if $\mathrm{dim}\,\mathrm{Ann}\left( \mathcal{T%
}\right) =1$, then $\mathcal{T}$ is a $1$-dimensional annihilator extension of $%
\mathcal{T}_{3,1}$ or $\mathcal{T}_{3,2}$.

Assume first that $\mathcal{T}$
is a $1$-dimensional annihilator extension of $\mathcal{T}_{3,1}$. In this
case, we have%
\begin{equation*}
Z^{3}\left( \mathcal{T}_{3,1},{\mathbb{F}}\right) =\left\langle
\Delta _{1,2,1},\Delta _{1,2,2},\Delta _{1,3,1},\Delta _{1,3,3},\Delta
_{2,3,2},\Delta _{2,3,3},\Delta _{1,2,3}+\Delta _{1,3,2},\Delta
_{2,3,1}+\Delta _{1,3,2}\right\rangle ,
\end{equation*}
and $B^{3}\left( \mathcal{T}_{3,2},{\mathbb{F}}\right) =0$. So $%
H^{3}\left( \mathcal{T}_{3,1},{\mathbb{F}}\right) =Z
^{3}\left( \mathcal{T}_{3,1},{\mathbb{F}}\right) $. Moreover, the
automorphism group $\mathrm{Aut}\left( \mathcal{T}_{3,1}\right)= GL({\mathcal{T}
_{3,1}})$.

Now, choose an arbitrary subspace $\mathbb{W}\in \mathcal{T}_{1}\left( \mathcal{T}
_{3,1}\right) $. Then $\mathbb{W}$ is spanned by%
\begin{equation*}
\theta =\alpha _{1,2,1}\Delta _{1,2,1}+\alpha _{1,3,1}\Delta _{1,3,1}+\alpha
_{2,3,1}\Delta _{2,3,1}+\alpha _{1,2,2}\Delta _{1,2,2}+\alpha _{1,3,2}\Delta
_{1,3,2}+\alpha _{2,3,2}\Delta _{2,3,2}+\alpha _{1,2,3}\Delta
_{1,2,3}+\alpha _{1,3,3}\Delta _{1,3,3}+\alpha _{2,3,3}\Delta _{2,3,3}
\end{equation*}
such that $\alpha _{1,3,2}=\alpha _{1,2,3}+\alpha _{2,3,1}$ and $\mathrm{Rad}
\left( \theta \right) \cap \mathcal{T}_{3,1}=0$. Further, $\mathrm{Rad}
\left( \theta \right) \cap \mathcal{T}_{3,1}=0$ if and only if the matrix $\Omega ^{t}$
has rank $3$ where%
\begin{equation*}
\Omega =%
\begin{pmatrix}
0 & 0 & 0 & \alpha _{1,2,1} & \alpha _{1,2,2} & \alpha _{1,2,3} & \alpha
_{1,3,1} & \alpha _{1,3,2} & \alpha _{1,3,3} \\
-\alpha _{1,2,1} & -\alpha _{1,2,2} & \alpha _{2,1,3} & 0 & 0 & 0 & \alpha
_{2,3,1} & \alpha _{2,3,2} & \alpha _{2,3,3} \\
-\alpha _{1,3,1} & -\alpha _{1,3,2} & -\alpha _{1,3,3} & -\alpha _{2,3,1} &
-\alpha _{2,3,2} & -\alpha _{2,3,3} & 0 & 0 & 0%
\end{pmatrix}
.
\end{equation*}
Let $\phi =\big(a_{ij}\big)\in $ $\mathrm{Aut}\left( \mathcal{T}
_{3,1}\right) $ and write%
\begin{equation*}
\phi \theta =\beta _{1,2,1}\Delta _{1,2,1}+\beta _{1,3,1}\Delta
_{1,3,1}+\beta _{2,3,1}\Delta _{2,3,1}+\beta _{1,2,2}\Delta _{1,2,2}+\beta
_{1,3,2}\Delta _{1,3,2}+\beta _{2,3,2}\Delta _{2,3,2}+\beta _{1,2,3}\Delta
_{1,2,3}+\beta _{1,3,3}\Delta _{1,3,3}+\beta _{2,3,3}\Delta _{2,3,3}.
\end{equation*}
Then%
\begin{eqnarray*}
\beta _{1,2,1} &=&\left( a_{11}a_{22}-a_{12}a_{21}\right) \left(
a_{11}\alpha _{1,2,1}+a_{21}\alpha _{1,2,2}+a_{31}\alpha _{1,2,3}\right)
+\left( a_{11}a_{32}-a_{12}a_{31}\right) \left( a_{11}\alpha
_{1,3,1}+a_{21}\alpha _{1,3,2}+a_{31}\alpha _{1,3,3}\right) \\
&&+\left( a_{21}a_{32}-a_{22}a_{31}\right) \left( a_{11}\alpha
_{2,3,1}+a_{21}\alpha _{2,3,2}+a_{31}\alpha _{2,3,3}\right) , \\
\beta _{1,3,1} &=&\left( a_{11}a_{23}-a_{21}a_{13}\right) \left(
a_{11}\alpha _{1,2,1}+a_{21}\alpha _{1,2,2}+a_{31}\alpha _{1,2,3}\right)
+\left( a_{11}a_{33}-a_{13}a_{31}\right) \left( a_{11}\alpha
_{1,3,1}+a_{21}\alpha _{1,3,2}+a_{31}\alpha _{1,3,3}\right) \\
&&+\left( a_{21}a_{33}-a_{31}a_{23}\right) \left( a_{11}\alpha
_{2,3,1}+a_{21}\alpha _{2,3,2}+a_{31}\alpha _{2,3,3}\right) , \\
\beta _{2,3,1} &=&\left( a_{12}a_{23}-a_{13}a_{22}\right) \left(
a_{11}\alpha _{1,2,1}+a_{21}\alpha _{1,2,2}+a_{31}\alpha _{1,2,3}\right)
+\left( a_{12}a_{33}-a_{13}a_{32}\right) \left( a_{11}\alpha
_{1,3,1}+a_{21}\alpha _{1,3,2}+a_{31}\alpha _{1,3,3}\right) \\
&&+\left( a_{22}a_{33}-a_{23}a_{32}\right) \left( a_{11}\alpha
_{2,3,1}+a_{21}\alpha _{2,3,2}+a_{31}\alpha _{2,3,3}\right) , \\
\beta _{1,2,2} &=&\left( a_{11}a_{22}-a_{12}a_{21}\right) \left(
a_{12}\alpha _{1,2,1}+a_{22}\alpha _{1,2,2}+a_{32}\alpha _{1,2,3}\right)
+\left( a_{11}a_{32}-a_{12}a_{31}\right) \left( a_{12}\alpha
_{1,3,1}+a_{22}\alpha _{1,3,2}+a_{32}\alpha _{1,3,3}\right) \\
&&+\left( a_{21}a_{32}-a_{22}a_{31}\right) \left( a_{12}\alpha
_{2,3,1}+a_{22}\alpha _{2,3,2}+a_{32}\alpha _{2,3,3}\right) , \\
\beta _{1,3,2} &=&\left( a_{11}a_{23}-a_{21}a_{13}\right) \left(
a_{12}\alpha _{1,2,1}+a_{22}\alpha _{1,2,2}+a_{32}\alpha _{1,2,3}\right)
+\left( a_{11}a_{33}-a_{13}a_{31}\right) \left( a_{12}\alpha
_{1,3,1}+a_{22}\alpha _{1,3,2}+a_{32}\alpha _{1,3,3}\right) \\
&&+\left( a_{21}a_{33}-a_{31}a_{23}\right) \left( a_{12}\alpha
_{2,3,1}+a_{22}\alpha _{2,3,2}+a_{32}\alpha _{2,3,3}\right) , \\
\beta _{2,3,2} &=&\left( a_{12}a_{23}-a_{13}a_{22}\right) \left(
a_{12}\alpha _{1,2,1}+a_{22}\alpha _{1,2,2}+a_{32}\alpha _{1,2,3}\right)
+\left( a_{12}a_{33}-a_{13}a_{32}\right) \left( a_{12}\alpha
_{1,3,1}+a_{22}\alpha _{1,3,2}+a_{32}\alpha _{1,3,3}\right) \\
&&+\left( a_{22}a_{33}-a_{23}a_{32}\right) \left( a_{12}\alpha
_{2,3,1}+a_{22}\alpha _{2,3,2}+a_{32}\alpha _{2,3,3}\right) , \\
\beta _{1,2,3} &=&\left( a_{11}a_{22}-a_{12}a_{21}\right) \left(
a_{13}\alpha _{1,2,1}+a_{23}\alpha _{1,2,2}+a_{33}\alpha _{1,2,3}\right)
+\left( a_{11}a_{32}-a_{12}a_{31}\right) \left( a_{13}\alpha
_{1,3,1}+a_{23}\alpha _{1,3,2}+a_{33}\alpha _{1,3,3}\right) \\
&&+\left( a_{21}a_{32}-a_{22}a_{31}\right) \left( a_{13}\alpha
_{2,3,1}+a_{23}\alpha _{2,3,2}+a_{33}\alpha _{2,3,3}\right) , \\
\beta _{1,3,3} &=&\left( a_{11}a_{23}-a_{21}a_{13}\right) \left(
a_{13}\alpha _{1,2,1}+a_{23}\alpha _{1,2,2}+a_{33}\alpha _{1,2,3}\right)
+\left( a_{11}a_{33}-a_{13}a_{31}\right) \left( a_{13}\alpha
_{1,3,1}+a_{23}\alpha _{1,3,2}+a_{33}\alpha _{1,3,3}\right) \\
&&+\left( a_{21}a_{33}-a_{31}a_{23}\right) \left( a_{13}\alpha
_{2,3,1}+a_{23}\alpha _{2,3,2}+a_{33}\alpha _{2,3,3}\right) , \\
\beta _{2,3,3} &=&\left( a_{12}a_{23}-a_{13}a_{22}\right) \left(
a_{13}\alpha _{1,2,1}+a_{23}\alpha _{1,2,2}+a_{33}\alpha _{1,2,3}\right)
+\left( a_{12}a_{33}-a_{13}a_{32}\right) \left( a_{13}\alpha
_{1,3,1}+a_{23}\alpha _{1,3,2}+a_{33}\alpha _{1,3,3}\right) \\
&&+\left( a_{22}a_{33}-a_{23}a_{32}\right) \left( a_{13}\alpha
_{2,3,1}+a_{23}\alpha _{2,3,2}+a_{33}\alpha _{2,3,3}\right) .
\end{eqnarray*}
Let us now associate  $\theta $  with another matrix $\mathcal{A}_{\theta }$
and $\phi \theta $ with  another matrix $\mathcal{A}_{\phi \theta }$ in the
following way:%
\begin{equation*}
\mathcal{A}_{\theta }\mathcal{=}
\begin{pmatrix}
\alpha _{2,3,1} & \alpha _{2,3,2} & \alpha _{2,3,3}\allowbreak \\
-\alpha _{1,3,1} & -\alpha _{1,3,2} & -\alpha _{1,3,3} \\
\alpha _{1,2,1} & \alpha _{1,2,2} & \alpha _{1,2,3}
\end{pmatrix}
,\mathcal{A}_{\phi \theta }\mathcal{=}
\begin{pmatrix}
\beta _{2,3,1} & \beta _{2,3,2} & \beta _{2,3,3}\allowbreak \\
-\beta _{1,3,1} & -\beta _{1,3,2} & -\beta _{1,3,3} \\
\beta _{1,2,1} & \beta _{1,2,2} & \beta _{1,2,3}
\end{pmatrix}
.
\end{equation*}
Then $\mathrm{tr}\left( \mathcal{A}_{\theta }\right) =\mathrm{tr}\left( \mathcal{A}_{\phi
\theta }\right) =0$. Moreover, we have $\mathcal{A}_{\phi \theta }\mathcal{=}
\det \left( \phi \right) \phi ^{-1}\mathcal{A}_{\theta }\phi $. Define%
\[
\mathcal{S}=\left\{ \mathcal{A}_{\theta }:\theta \in Z^{3}\left( \mathcal{T}
_{3,1},{\mathbb{F}}\right) \text{ such that }\mathrm{Rad}\left( \theta
\right) \cap \mathcal{T}_{3,1}=0\right\} .
\]%
Then $\mathrm{Aut}\left( \mathcal{T}_{3,1}\right) $ acts on $\mathcal{S}$ by 
$\phi \bullet \mathcal{A}_{\theta }=\mathcal{A}_{\phi \theta }$. It follows
that any element in $\mathcal{S}$ is in the same orbit as one of the
following:

\begin{tasks}(2)
\task[(i)] $%
\begin{pmatrix}
0 & 1 & 0 \\ 
0 & 0 & 1 \\ 
0 & 0 & 0%
\end{pmatrix}
.$

\task[(ii)] $%
\begin{pmatrix}
\lambda  & 1 & 0 \\ 
0 & \lambda  & 0 \\ 
0 & 0 & -2\lambda 
\end{pmatrix}
:\lambda \neq 0.$

\task[(iii)] $%
\begin{pmatrix}
\lambda  & 0 & 0 \\ 
0 & -\lambda  & 0 \\ 
0 & 0 & 0%
\end{pmatrix}
:\lambda \neq 0.$

\task[(iv)] $%
\begin{pmatrix}
\lambda  & 0 & 0 \\ 
0 & \mu  & 0 \\ 
0 & 0 & -\left( \lambda +\mu \right) 
\end{pmatrix}
:\allowbreak \lambda \mu \left( \lambda +\mu \right) \neq 0.$
\end{tasks}

\bigskip

We will study every case individually.
\begin{itemize}
    \item If $\mathcal{A}_{\theta }$ has the form $($i$)$, then $\theta =\Delta
_{2,3,2}-\Delta _{1,3,3}$. So we get the Lie triple system $\mathcal{T}_{4,4}
$. 
    
    \item If $\mathcal{A}_{\theta }$ has the form $($ii$)$, then%
\[
\vartheta =\lambda ^{-1}\theta =\Delta _{2,3,1}-\Delta _{1,3,2}-2\Delta
_{1,2,3}+\lambda \Delta _{2,3,2}\text{.}
\]%
Since 
\[
\mathcal{A}_{\vartheta }=%
\begin{pmatrix}
1 & 0 & 0 \\ 
0 & \lambda ^{-1} & 0 \\ 
0 & 0 & 1%
\end{pmatrix}
\begin{pmatrix}
1 & 1 & 0 \\ 
0 & 1 & 0 \\ 
0 & 0 & -2%
\end{pmatrix}
\begin{pmatrix}
1 & 0 & 0 \\ 
0 & \lambda  & 0 \\ 
0 & 0 & 1%
\end{pmatrix}
,
\]%
the matrix $\mathcal{A}_{\vartheta }$ is similar to:%
\[
\begin{pmatrix}
1 & 1 & 0 \\ 
0 & 1 & 0 \\ 
0 & 0 & -2%
\end{pmatrix}
.
\]%
So we may assume that $\vartheta =\Delta _{2,3,1}-\Delta _{1,3,2}-2\Delta
_{1,2,3}+\Delta _{2,3,2}$. Since $\mathbb{W=}\left\langle \theta
\right\rangle =\left\langle \vartheta \right\rangle $, we get the triple
system $\mathcal{T}_{4,5}$. 
    
    \item If $\mathcal{A}_{\theta }$ has the form $($iii$)$, we
may assume that $\theta =-\Delta _{1,3,2}- \Delta
_{1,2,3}$. Then we get the Lie triple system $\mathcal{T}_{4,6}^{0}$.
    
    \item If $%
\mathcal{A}_{\theta }$ has the form $($iv$)$, then we may assume that $\theta
=\lambda \Delta _{2,3,1}-\Delta _{1,3,2}-\left( \lambda +1\right) \Delta
_{1,2,3}$ such that $\lambda \left( \lambda +1\right) \neq 0$. So we get the
Lie triple systems $\mathcal{T}_{4,6}^{\lambda }\left( \lambda ^{2}+\lambda \neq
0\right) $. Further, if $\mathcal{T}_{4,6}^{\lambda }\cong \mathcal{T}
_{4,6}^{\lambda ^{\prime }}$, then there exists a non-zero element $\mu \in
\mathbb{F}$ such that the following matrices have the same eigenvalues:%
\begin{equation*}
\begin{pmatrix}
\lambda  & 0 & 0 \\
0 & 1 & 0 \\
0 & 0 & -\left( \lambda +1\right)
\end{pmatrix}
,\mu
\begin{pmatrix}
\lambda ^{\prime } & 0 & 0 \\
0 & 1 & 0 \\
0 & 0 & -\left( \lambda ^{\prime }+1\right)
\end{pmatrix}
.
\end{equation*}
So we have $\left\{ 1,\lambda ,-\left( \lambda +1\right) \right\} =\left\{
\mu ,\mu \lambda ^{\prime },-\mu \left( \lambda ^{\prime }+1\right) \right\}
$ and hence, $\lambda ^{\prime }\in \left\{ \lambda ,-\left( \lambda
+1\right) ,\frac{1}{\lambda },-\frac{\lambda +1}{\lambda },-\frac{1}{\lambda
+1},-\frac{\lambda }{\lambda +1}\right\} $. Conversely, if $\lambda ^{\prime
}\in \left\{ \lambda ,-\left( \lambda +1\right) ,\frac{1}{\lambda },-\frac{%
\lambda +1}{\lambda },-\frac{1}{\lambda +1},-\frac{\lambda }{\lambda +1}
\right\} $, we have the following isomorphisms:%
\begin{eqnarray*}
\sigma _{1} &:&\mathcal{T}_{4,6}^{\lambda }\longrightarrow \mathcal{T}_{4,6}^{\lambda }:\sigma _{1}\left( e_{1}\right) =e_{1},\sigma _{1}\left( e_{2}\right)
=e_{2},\sigma _{1}\left( e_{3}\right) =e_{3},\sigma _{1}\left( e_{4}\right)
=e_{4}, \\
\sigma _{2} &:&\mathcal{T}_{4,6}^{\lambda }\longrightarrow \mathcal{T}
_{4,6}^{-\left( \lambda +1\right) }:\sigma _{2}\left( e_{1}\right)
=e_{3},\sigma _{2}\left( e_{2}\right) =e_{2},\sigma _{2}\left( e_{3}\right)
=e_{1},\sigma _{2}\left( e_{4}\right) =-e_{4}, \\
\sigma _{3} &:&\mathcal{T}_{4,6}^{\lambda }\longrightarrow \mathcal{T}
_{4,6}^{\frac{1}{\lambda }}:\sigma _{3}\left( e_{1}\right) =e_{2},\sigma
_{3}\left( e_{2}\right) =e_{1},\sigma _{3}\left( e_{3}\right) =e_{3},\sigma
_{1}\left( e_{4}\right) =-\frac{1}{\lambda }e_{4}, \\
\sigma _{4} &:&\mathcal{T}_{4,6}^{\lambda }\longrightarrow \mathcal{T}
_{4,6}^{-\frac{\lambda +1}{\lambda }}:\sigma _{4}\left( e_{1}\right)
=e_{2},\sigma _{4}\left( e_{2}\right) =e_{3},\sigma _{4}\left( e_{3}\right)
=e_{1},\sigma _{4}\left( e_{4}\right) =\frac{1}{\lambda }e_{4}, \\
\sigma _{5} &:&\mathcal{T}_{4,6}^{\lambda }\longrightarrow \mathcal{T}
_{4,6}^{-\frac{1}{\lambda +1}}:\sigma _{5}\left( e_{1}\right) =e_{3},\sigma
_{5}\left( e_{2}\right) =e_{1},\sigma _{5}\left( e_{3}\right) =e_{2},\sigma
_{5}\left( e_{4}\right) =-\frac{1}{\lambda +1}e_{4}, \\
\sigma _{6} &:&\mathcal{T}_{4,6}^{\lambda }\longrightarrow \mathcal{T}
_{4,6}^{-\frac{\lambda }{\lambda +1}}:\sigma _{6}\left( e_{1}\right)
=e_{1},\sigma _{6}\left( e_{2}\right) =e_{3},\sigma _{6}\left( e_{3}\right)
=e_{2},\sigma _{6}\left( e_{4}\right) =\frac{1}{\lambda +1}e_{4}.
\end{eqnarray*}

Moreover, from here we have that $\mathcal{T}_{4,6}^{0} \cong \mathcal{T}_{4,6}^{-1}$.

Now, for any non-zero elements $\alpha ,\beta ,\gamma \in \mathbb{F}$, we define $%
\kappa \left( \alpha ,\beta ,\gamma \right) =\frac{\left( \alpha ^{2}+\beta
^{2}+\gamma ^{2}\right) ^{3}}{8\alpha ^{2}\beta ^{2}\gamma ^{2}}$. If $%
\left\{ \alpha ,\beta ,\gamma \right\} =\left\{ \mu \alpha ^{\prime },\mu
\beta ^{\prime },\mu \gamma ^{\prime }\right\} $, then $\kappa \left( \alpha
,\beta ,\gamma \right) =\kappa \left( \alpha ^{\prime },\beta ^{\prime
},\gamma ^{\prime }\right) $. Thus,%
\begin{equation*}
\xi \left( \lambda \right) =\kappa \left( 1,\lambda ,-\left( \lambda
+1\right) \right) =\frac{\left( \lambda ^{2}+\lambda +1\right) ^{3}}{\lambda
^{2}\left( \lambda +1\right) ^{2}}:\lambda \neq -1,0
\end{equation*}
is an invariant for $\mathcal{T}_{4,6}^{\lambda }$. If $\xi \left( \lambda
\right) =\xi \left( \lambda ^{\prime }\right) $, then $\lambda ^{\prime }\in
\left\{ \lambda ,-\frac{1}{\lambda +1},\frac{1}{\lambda },-\lambda -1,-\frac{%
\lambda }{\lambda +1},-\frac{1}{\lambda }\left( \lambda +1\right) \right\} $%
. Hence, $\mathcal{T}_{4,6}^{\lambda }\cong \mathcal{T}_{4,6}^{\lambda
^{\prime }}$ if and only if $\allowbreak \xi \left( \lambda \right) =\xi \left( \lambda
^{\prime }\right) $. 
\end{itemize}

Now, let us assume that $\mathcal{T}$
is a $1$-dimensional annihilator extension of $\mathcal{T}_{3,2}$. In this case, we have%
\begin{equation*}
Z^{3}\left( \mathcal{T}_{3,2},{\mathbb{F}}\right) =\left\langle
\Delta _{1,2,1},\Delta _{1,2,2},\Delta _{1,3,1},\Delta _{1,2,3}+\Delta
_{1,3,2}\right\rangle
\end{equation*}
and $B^{3}\left( \mathcal{T}_{3,2},{\mathbb{F}}\right)
=\left\langle \Delta _{1,2,1}\right\rangle $. Therefore, $H^{3}\left(
\mathcal{T}_{3,2},{\mathbb{F}}\right) =\left\langle \left[ \Delta _{1,2,2}
\right] ,\left[ \Delta _{1,3,1}\right] ,\left[ \Delta _{1,2,3}\right] +\left[
\Delta _{1,3,2}\right] \right\rangle $. Moreover, by Corollary \ref{Aut}, the
automorphism group $\mathrm{Aut}\left( \mathcal{T}_{3,2}\right) $ consists
of invertible matrices of the form%
\begin{equation*}
\phi =%
\begin{pmatrix}
a_{11} & 0 & 0 \\
a_{21} & a_{22} & 0 \\
a_{31} & a_{32} & a_{11}^{2}a_{22}
\end{pmatrix}
.
\end{equation*}
Choose an arbitrary subspace $\mathbb{W}\in \mathcal{T}_{1}\left( \mathcal{T}
_{3,2}\right) $. Then $\mathbb{W}$ is spanned by%
\begin{equation*}
\left[ \theta \right] =\alpha _{1,3,1}\left[ \Delta _{1,3,1}\right] +\alpha
_{1,2,2}\left[ \Delta _{1,2,2}\right] +\alpha _{1,2,3}\left( \left[ \Delta
_{1,2,3}\right] +\left[ \Delta _{1,3,2}\right] \right)
\end{equation*}
such that $\mathrm{Rad}\left( \theta \right) \cap \left\langle
e_{3}\right\rangle =0$. Let $\phi =\big(a_{ij}\big)\in $ $\mathrm{Aut}\left(
\mathcal{T}_{3,2}\right) $ and write $\phi \theta =\beta _{1,3,1}\left[
\Delta _{1,3,1}\right] +\beta _{1,2,2}\left[ \Delta _{1,2,2}\right] +\beta
_{1,2,3}\left( \left[ \Delta _{1,2,3}\right] +\left[ \Delta _{1,3,2}\right]
\right) $. Then%
\begin{eqnarray*}
\beta _{1,3,1} &=&a_{11}^{3}a_{22}\left( a_{11}\alpha _{1,3,1}+a_{21}\alpha
_{1,2,3}\right) , \\
\beta _{1,2,2} &=&a_{11}a_{22}\left( a_{22}\alpha _{1,2,2}+2a_{32}\alpha
_{1,2,3}\right) , \\
\beta _{1,2,3} &=&\det \left( \phi \right) \alpha _{1,2,3}.
\end{eqnarray*}
Since $\beta _{1,2,3}\neq 0$ if and only if $\alpha _{1,2,3}\neq 0$, we have \emph{$%
\mathcal{O}$}$\left\langle \left[ \theta \right] :\alpha _{1,2,3}\neq
0\right\rangle \cap $\emph{$\mathcal{O}$}$\left\langle \left[ \theta \right]
:\alpha _{1,2,3}=0\right\rangle =\emptyset $ and therefore $\mathrm{Aut}
\left( \mathcal{T}_{3,2}\right) $ has at least two orbits on $\mathcal{T}
_{1}\left( \mathcal{T}_{3,2}\right) $. We distinguish two cases:

\begin{itemize}
\item $\alpha _{1,2,3}\neq 0$. So we may assume $\alpha _{1,2,3}=1$. Let $%
\phi $ be the following automorphism:%
\begin{equation*}
\phi =%
\begin{pmatrix}
1 & 0 & 0 \\
-\alpha _{1,3,1} & 1 & 0 \\
0 & -\frac{1}{2}\alpha _{1,2,2} & 1%
\end{pmatrix}
.
\end{equation*}
Then $\phi \mathbb{W=}\left\langle \left[ \Delta _{1,2,3}\right] +\left[
\Delta _{1,3,2}\right] \right\rangle $. So we get the Lie triple system $\mathcal{T}
_{4,7}$.

\item $\alpha _{1,2,3}=0$. Then $\mathrm{Rad}\left( \theta \right) \cap
\left\langle e_{3}\right\rangle =0$ implies that $\alpha _{1,3,1}\neq 0$ and
so we may assume $\alpha _{1,3,1}=1$. From here, we have%
\begin{equation*}
\beta _{1,3,1}=a_{11}^{4}a_{22},\beta _{1,2,2}=a_{11}a_{22}^{2}\alpha
_{1,2,2},\beta _{1,2,3}=0.
\end{equation*}
This in turn shows that if $\alpha _{1,2,3}=0$ then \emph{$\mathcal{O}$}$%
\left\langle \left[ \theta \right] :\alpha _{1,2,2}\neq 0\right\rangle \cap $%
\emph{$\mathcal{O}$}$\left\langle \left[ \theta \right] :\alpha
_{1,2,2}=0\right\rangle =\emptyset $. We distinguish two subcases:

\begin{itemize}
\item $\mu =\alpha _{1,2,2}\neq 0$. Let $\phi $ be the following
automorphism:%
\begin{equation*}
\phi =%
\begin{pmatrix}
\mu  & 0 & 0 \\
0 & \mu ^{2} & 0 \\
0 & 0 & \mu ^{4}
\end{pmatrix}
.
\end{equation*}
Then $\left[ \phi \theta \right] =\mu ^{6}\left( \left[ \Delta _{1,3,1}
\right] +\left[ \Delta _{1,2,2}\right] \right) $. Hence, $\phi \mathbb{W=}
\left\langle \left[ \Delta _{1,3,1}\right] +\left[ \Delta _{1,2,2}\right]
\right\rangle $ and we get the Lie triple system $\mathcal{T}_{4,8}$.

\item $\alpha _{1,2,2}=0$. Then $\mathbb{W=}\left\langle \left[ \Delta
_{1,3,1}\right] \right\rangle $ and we get the Lie triple system $\mathcal{T}_{4,9}$.

\end{itemize}
\end{itemize}

\end{proof}

\section{The geometric classification of nilpotent Lie triple systems up to dimension four}

Given a vector space ${\mathbb V}$ of dimension $n$, the set of trilinear maps $\textrm{Tri}({\mathbb V} , {\mathbb V}) \cong \textrm{Hom}({\mathbb V} ^{\otimes3}, {\mathbb V})\cong ({\mathbb V}^*)^{\otimes3} \otimes {\mathbb V}$ is a vector space of dimension $n^4$. This vector space has the structure of the affine space ${\mathbb F}^{n^4}$ in the following sense:
fixed a basis $e_1, \ldots, e_n$ of ${\mathbb V}$, then any triple system with multiplication $\mu \in \textrm{Tri}({\mathbb V} , {\mathbb V})$, is determined by some parameters $c_{ijk}^p \in {\mathbb F}$,  called {\it structural constants},  such that
$$\mu(e_i, e_j, e_k) = \sum_{p=1}^n c_{ijk}^p e_p$$
which corresponds to a point in the affine space ${\mathbb F}^{n^4}$. Then a set of triple systems $\mathcal S$ corresponds to an algebraic variety, i.e., a Zariski closed set, if there are some polynomial equations in variables $c_{ijk}^p$ with zero locus equal to the set of structural constants of the triple systems in $\mathcal S$. Since given the identities defining Lie triple systems we can obtain a set of polynomial equations in variables $c_{ijk}^p$, the class of $n$-dimensional Lie triple systems is a variety. Moreover, the class of $n$-dimensional nilpotent Lie triple systems, denoted $\mathcal{T}_{n}$, is a subvariety of the variety of $n$-dimensional Lie triple systems.

Now, consider the action of $\textrm{GL}({\mathbb V})$ on $\mathcal{T}_{n}$ by conjugation:
$$(g*\mu)(x,y,z) = g \mu (g^{-1} x, g^{-1} y, g^{-1} z)$$
for $g\in\textrm{GL}({\mathbb V})$, $\mu\in \mathcal{T}_{n}$ and for any $x, y, z \in {\mathbb V}$. Observe that the $\textrm{GL}({\mathbb V})$-orbit of $\mu$, denoted $O(\mu)$, contains all the structural constants of the triple systems isomorphic to the Lie triple system with structural constants $\mu$.

In the previous section, we gave a decomposition of $\mathcal{T}_{n}$, for $n\leq 4$, into $\textrm{GL}(\mathbb V)$-orbits (acting by conjugation), i.e., an algebraic classification of the nilpotent Lie triple systems up to dimension four. In this section, we will describe the closures of orbits of $\mu\in{\mathcal{T}_{n}}$, denoted by $\overline{O(\mu)}$, and we will give a geometric classification of ${\mathcal{T}_{n}}$, which consist in describing its irreducible components.
Recall that any affine variety can be represented as a finite union of its irreducible components in a unique way.

Additionally, describing the irreducible components of a variety, such as ${\mathcal{T}_{n}}$, gives us the rigid triple systems of the variety, which are those triple systems with an open $\textrm{GL}(\mathbb V)$-orbit. This is due to the  fact that a triple system is rigid in variety if and only if the closure of its orbit is an irreducible component of the variety.

\begin{definition}
\rm Let $\mathcal  T $ and $\mathcal  T '$ be two $n$-dimensional Lie triple systems and $\mu, \lambda \in \mathcal{T}_{n}$ be their representatives in the affine space, respectively. We say $\mathcal  T $ {\it degenerates}  to $\mathcal  T '$, and write $\mathcal  T \to \mathcal  T '$, if $\lambda\in\overline{O(\mu)}$. If $\mathcal  T  \not\cong \mathcal  T '$, then we call it a  {\it proper degeneration}.

Conversely, if $\lambda\not\in\overline{O(\mu)}$ then we call it a 
{\it non-degeneration} and we write ${\mathcal  T }\not\to {\mathcal  T }'$.
\end{definition}

\noindent Note that the definition of a degeneration does not depend on the choice of $\mu$ and $\lambda$. Also, due to the transitivity of the notion of degeneration  (that is, if ${\mathcal  T }\to {\mathcal  T }''$ and ${\mathcal  T }''\to {\mathcal  T }'$ then ${\mathcal  T }\to{\mathcal  T }'$) we have the following definitions.

\begin{definition} \rm
Let $\mathcal  T $ and $\mathcal  T '$ be two $n$-dimensional Lie triple systems such that ${\mathcal  T } \to {\mathcal  T }'$. If there is no ${\mathcal  T }''$ such that ${\mathcal  T }\to {\mathcal  T }''$ and ${\mathcal  T }''\to {\mathcal  T }'$ are proper degenerations, then ${\mathcal  T }\to {\mathcal  T }'$ is called a  {\it primary degeneration}. Analogously, let $\mathcal  T $ and $\mathcal  T '$ be two $n$-dimensional Lie triple systems such that $\mathcal  T \not\to \mathcal  T '$, if there are no ${\mathcal  T }''$ and ${\mathcal  T }'''$ such that ${\mathcal  T }''\to {\mathcal  T }$, ${\mathcal  T }'\to {\mathcal  T }'''$, ${\mathcal  T }''\not\to {\mathcal  T }'''$ and one of the assertions ${\mathcal  T }''\to {\mathcal  T }$ and ${\mathcal  T }'\to {\mathcal  T }'''$ is a proper degeneration,  then ${\mathcal  T } \not\to {\mathcal  T }'$ is called a  {\it primary non-degeneration}.
\end{definition}

Note that it suffices to prove primary degenerations and non-degenerations to fully describe the geometry of a variety. Therefore, in this work we will focus on proving the primary degenerations and non-degenerations, and for this we will use the results applied to Lie algebras in \cite{BC99,GRH,GRH2,S90}. These results are summarized below.

Firstly, since $\mathrm{dim}\,O(\mu) = n^2 - \mathrm{dim}\,\mathfrak{Der}(\mu)$, then if $ \mathcal  T \to  \mathcal  T '$ and  $\mathcal  T \not\cong  \mathcal  T '$, we have that $\mathrm{dim}\,\mathfrak{Der}( \mathcal  T )<\mathrm{dim}\,\mathfrak{Der}( \mathcal  T ')$, where $\mathfrak{Der}( \mathcal  T )$ denotes the Lie algebra of derivations of  $\mathcal  T $. Therefore, we will check the assertion ${\mathcal  T }\to {\mathcal  T }'$ only for ${\mathcal  T }$ and ${\mathcal  T }'$ such that $\mathrm{dim}\,\mathfrak{Der}({\mathcal  T })<\mathrm{dim}\,\mathfrak{Der}({\mathcal  T }')$.

Secondly, to prove primary degenerations, let ${\mathcal  T }$ and ${\mathcal  T }'$ be two Lie triple systems represented by the structures $\mu$ and $\lambda$ from ${\mathcal  T }_n$, respectively. Let $c_{ijk}^p$ be the structure constants of $\lambda$ in a basis $e_1,\dots, e_n$ of ${\mathbb V}$. If there exist $n^2$ maps $a_i^j(t): \mathbb{F}^*\to \mathbb{F}$ such that $E_i(t)=\sum_{j=1}^na_i^j(t)e_j$ ($1\leq i \leq n$) form a basis of ${\mathbb V}$ for any $t\in\mathbb{F}^*$ and the structure constants $c_{ijk}^p(t)$ of $\mu$ in the basis $E_1(t),\dots, E_n(t)$ satisfy $\lim\limits_{t\to 0}c_{ijk}^p(t)=c_{ijk}^p$, then ${\mathcal  T }\to {\mathcal  T }'$. In this case,  $E_1(t),\dots, E_n(t)$ is called a parametrized basis for ${\mathcal  T }\to {\mathcal  T }'$.

Thirdly, to prove primary non-degenerations we will use the following lemma (see \cite{GRH}).

\begin{lemma}\label{main1}
Let $\mathcal{B}$ be a Borel subgroup of ${\rm GL}({\mathbb V})$ and $\mathcal{R}\subset {\mathcal  T }_n$ be a $\mathcal{B}$-stable closed subset.
If ${\mathcal  T } \to {\mathcal  T }'$ and ${\mathcal  T }$ can be represented by $\mu\in\mathcal{R}$ then there is $\lambda\in \mathcal{R}$ that represents ${\mathcal  T }'$.
\end{lemma}

\noindent In particular, it follows:

\begin{corollary}\label{cor:degs}
Let ${\mathcal  T }\to {\mathcal  T }'$ be a proper degeneration, then
\begin{enumerate}
    \item $\mathrm{dim}\, \mathrm{Ann}({\mathcal  T })\leq\mathrm{dim}\,\mathrm{Ann}({\mathcal  T }')$, 
    \item $\mathrm{dim}\,\left[{\mathcal  T }, {\mathcal  T }, {\mathcal  T } \right]\geq \mathrm{dim}\,\left[{\mathcal  T }' ,{\mathcal  T }', {\mathcal  T }'\right]$.
\end{enumerate}
\end{corollary}

Defining $\mathcal{R}$ by a set of polynomial equations in variables $c_{ijk}^p$ and in the conditions of the previous result, such that $\mu\in\mathcal{R}$ and $O(\lambda)\cap \mathcal{R}=\emptyset$, give us the non-degeneration ${\mathcal  T }\not\to {\mathcal  T }'$. In this case, we call $\mathcal{R}$ a separating set for ${\mathcal  T }\not\to {\mathcal  T }'$.
To prove non-degenerations, we will present the corresponding separating set and we will omit the verification of the fact that $\mathcal{R}$ is stable under the action of the Borel subgroup of lower triangular matrices and of the fact that $O(\lambda)\cap \mathcal{R}=\emptyset$, which can be obtained by straightforward calculations.

Finally, if the algebraic classification of the class under consideration is finite, then the graph of primary degenerations gives the whole geometric classification: the description of irreducible components can be easily obtained. This is the case of the variety $\mathcal{T}_{3}$ of $3$-dimensional nilpotent Lie triple systems. However, the variety $\mathcal{T}_{4}$ of $4$-dimensional nilpotent Lie triple systems contains infinitely many non-isomorphic triple systems, since we have the family $\mathcal{T}_{4,6}^{*}$. Therefore, we have to fulfil some additional work.

\begin{definition}
\rm
Let ${\mathcal  T }(*)=\{{\mathcal  T }(\alpha): {\alpha\in I}\}$ be a family of $n$-dimensional Lie triple systems and let ${\mathcal  T }'$ be another Lie triple system. Suppose that ${\mathcal  T }(\alpha)$ is represented by the structure $\mu(\alpha)\in{\mathcal  T }_n$ for $\alpha\in I$ and ${\mathcal  T }'$ is represented by the structure $\lambda\in{\mathcal  T }_n$. We say the family ${\mathcal  T }(*)$ {\it degenerates}   to ${\mathcal  T }'$, and write ${\mathcal  T }(*)\to {\mathcal  T }'$, if $\lambda\in\overline{\{O(\mu(\alpha))\}_{\alpha\in I}}$.

Conversely, if $\lambda\not\in\overline{\{O(\mu(\alpha))\}_{\alpha\in I}}$ then we call it a  {\it non-degeneration}, and we write ${\mathcal  T }(*)\not\to {\mathcal  T }'$.

\end{definition}

To prove ${\mathcal  T }(*)\to {\mathcal  T }'$, suppose that ${\mathcal  T }(\alpha)$ is represented by the structure $\mu(\alpha)\in{\mathcal  T }_n$ for $\alpha\in I$ and ${\mathcal  T }'$ is represented by the structure $\lambda\in{\mathcal  T }_n$. Let $c_{ijk}^p$ be the structure constants of $\lambda$ in a basis  $e_1,\dots, e_n$ of ${\mathbb V}$. If there is a pair of maps $(f, (a_i^j))$, where $f:\mathbb{F}^*\to I$ and $a_i^j:\mathbb{F}^*\to \mathbb{F}$ are such that $E_i(t)=\sum_{j=1}^na_i^j(t)e_j$ ($1\le i\le n$) form a basis of ${\mathbb V}$ for any  $t\in\mathbb{F}^*$ and the structure constants $c_{ijk}^p(t)$ of $\mu\big(f(t)\big)$ in the basis $E_1(t),\dots, E_n(t)$ satisfy $\lim\limits_{t\to 0}c_{ijk}^p(t)=c_{ijk}^p$, then ${\mathcal  T }(*)\to {\mathcal  T }'$. In this case  $E_1(t),\dots, E_n(t)$ and $f(t)$ are called a parametrized basis and a parametrized index for ${\mathcal  T }(*)\to {\mathcal  T }'$, respectively.

To prove ${\mathcal  T }(*)\not \to {\mathcal  T }'$, we will use an analogue of the Lemma \ref{main1} for families of triple systems.

\begin{lemma}\label{main2}
Let $\mathcal{B}$ be a Borel subgroup of ${\rm GL}({\mathbb V})$ and $\mathcal{R}\subset {\mathcal  T }_n$ be a $\mathcal{B}$-stable closed subset.
If ${\mathcal  T }(*) \to {\mathcal  T }'$ and ${\mathcal  T }(\alpha)$ can be represented by $\mu(\alpha)\in\mathcal{R}$ for $\alpha\in I$, then there is $\lambda\in \mathcal{R}$ that represents ${\mathcal  T }'$.
\end{lemma}

Similarly, constructing a set $\mathcal{R}$ in the conditions of the previous result, such that $\mu(\alpha)\in\mathcal{R}$ for any $\alpha\in I$ and $O(\lambda)\cap \mathcal{R}=\varnothing$, gives us the non-degeneration ${\mathcal  T }(*)\not\to {\mathcal  T }'$.

\subsection{Geometric classification of 3-dimensional nilpotent Lie triple systems}

Since any triple system degenerates to the trivial triple system, we have the following result.

\begin{theorem}
The variety of $3$-dimensional nilpotent Lie triple systems has only one irreducible component, corresponding to the rigid Lie triple system $\mathcal{T}_{3,2}$. Moreover, $\mathrm{dim}\, O(\mathcal{T}_{3,2})=4$.
\end{theorem}

\subsection{Geometric classification of 4-dimensional nilpotent Lie triple systems}

In Theorem \ref{classification4} we present the classification, up to isomorphism, of the nilpotent Lie triple systems of dimension $4$ over an
algebraically closed field $\mathbb{F}$ of characteristic different from two and three. To obtain the irreducible components of this variety, we have to study the primary degenerations and non-degenerations first.

\begin{lemma} \label{th:deg4nlts}
The graph of primary degenerations and non-degenerations for the variety of $4$-dimensional nilpotent Lie triple systems is given in Figure 1, where the numbers above are the dimensions of the corresponding orbits.
\end{lemma}
\begin{proof}
From Table \ref{tab:clas4} we deduce the dimensions of the orbits for each nilpotent Lie triple system of dimension $4$. Every primary degeneration and non-degeneration can be proven using the parametrized bases and separating sets included in Table \ref{tab:algdeg} and Table \ref{tab:algndeg} below, respectively.
\end{proof}

\begin{center}
{
\begin{tikzpicture}[->,>=stealth,shorten >=0.05cm,auto,node distance=1.3cm,
	thick,main node/.style={rectangle,draw,fill=gray!10,rounded corners=1.5ex,font=\sffamily \scriptsize \bfseries },rigid node/.style={rectangle,draw,fill=black!20,rounded corners=1.5ex,font=\sffamily \scriptsize \bfseries },style={draw,font=\sffamily \scriptsize \bfseries }]

	\node (11) at (-1,0) {$11$};
	\node (10) at (2,0) {$10$};
	\node (9) at (5,0) {$9$};
	\node (8) at (7,0) {$8$};
	\node (5) at (9,0) {$5$};
	\node (0)  at (11,0) {$0$};

	
	\node[main node] (c48) at (-1,-2.5) {$\mathcal{T}_{4,7}$};

	\node[main node] (c45) at (2,-1) {$\mathcal{T}_{4,5}$};
	\node[main node] (c47l) at (2,-2.5) {$\mathcal{T}_{4,6}^{\lambda}$};
	\node[main node] (c49) at (2,-4) {$\mathcal{T}_{4,8}$};

	\node[main node] (c44) at (5,-1) {$\mathcal{T}_{4,4}$};
	\node[main node] (c410) at (5, -4) {$\mathcal{T}_{4,9}$};

	\node[main node] (c43) at (7,-1) {$\mathcal{T}_{4,3}$};
	\node[main node] (c47) at (7, -4) {$\mathcal{T}_{4,6}^{1}$};

	\node[main node] (c42) at (9,-2.5) {$\mathcal{T}_{4,2}$};

	\node[main node] (c41) at (11,-2.5) {$\mathcal{T}_{4,1}$};

    \path[every node/.style={font=\sffamily\small}]

    (c48) edge node[above=0, right=-15, fill=white]{\tiny\arraycolsep=0.3pt $\begin{array}{c}
         \lambda=0 \\
         \lambda=-1 
    \end{array}$ } (c47l) 
    (c48) edge (c49)
    (c45) edge (c47)
    (c45) edge (c44)
    (c49) edge (c410)
    (c49) edge (c43)
    (c49) edge (c44)

    (c44) edge (c42)

    (c410) edge (c42)
    (c43) edge (c42)
    (c42) edge (c41)
    (c47) edge (c42)
    (c47l) edge (c44);

\end{tikzpicture}}

{Figure 1.}  Graph of primary degenerations and non-degenerations.
\end{center}

At this point, only the description of the closure of the orbit of the parametric family $\mathcal{T}_{4,6}^{*}$  is missing.

\begin{lemma}\label{th:deg4nltsf}
The closure of the orbit of the parametric family $\mathcal{T}_{4,6}^{*}$ in the variety ${\mathcal{T}_{4}}$ contains the closures of the orbits of the nilpotent Lie triple systems $\mathcal{T}_{4,5}$, $\mathcal{T}_{4,4}$, $\mathcal{T}_{4,2}$ and $\mathcal{T}_{4,1}$.
\end{lemma}

\begin{proof}
The primary degenerations and non-degenerations that do not follow from the previous results are included in Table \ref{tab:infseriesdeg} and Table \ref{tab:infseriesndeg}  below, respectively.
\end{proof}

From Lemma \ref{th:deg4nlts} and Lemma \ref{th:deg4nltsf}, we have the following result that summarizes the geometric classification of ${\mathcal{T}_{4}}$.

\begin{theorem}
The variety of $4$-dimensional nilpotent Lie triple systems ${\mathcal{T}_{4}}$ has two irreducible components corresponding to the rigid Lie triple system $\mathcal{T}_{4,7}$ and the family of Lie triple systems $\mathcal{T}_{4,6}^{*}$.
The lattice of subsets for the orbit closures is given in Figure 2, where the numbers above are the dimensions of the corresponding orbits.
\end{theorem}

\begin{center}\label{lattice}
{\tiny
\begin{tikzpicture}[-,draw=gray!50,node distance=0.93cm,
                   ultra thick,main node/.style={rectangle, fill=gray!50,font=\sffamily \scriptsize \bfseries },style={draw,font=\sffamily \scriptsize \bfseries }]

\node (6) {$11$};
\node (6r) [right  of=6] {};
\node (6rr) [right  of=6r] {};
\node (6rrr) [right  of=6rr] {};
\node (5) [right  of=6rr]      {$10$};
\node (5r) [right  of=5] {};
\node (5rr) [right  of=5r] {};
\node (5rrr) [right  of=5rr] {};
\node (4) [right  of=5rrr]      {$9$};
\node (4r) [right  of=4] {};
\node (4rr) [right  of=4r] {};
\node (3) [right  of=4rr]      {$8$};
\node (3r) [right  of=3] {};
\node (3rr) [right  of=3r] {};
\node (2) [right  of=3rr]      {$5$};
\node (2r) [right  of=2] {};
\node (2rr) [right  of=2r] {};
\node (2rrr) [right  of=2rr] {};
\node (1) [right  of=2rr]      {$0$};
\node (1r) [right  of=1] {};
\node (1rr) [right  of=1r] {};
\node (1rrr) [right  of=1rr] {};

	\node (11p1)  [below of =6]     	{};
	\node (11p2)  [below of =11p1]     	{};
	\node (11p3)  [below of =11p2]     	{};

	\node[main node] (c47l)  [below of =11p1]     	{$\overline{O\big(\mathcal{T}_{4,6}^{*}\big)}$};	
	\node[main node] (c48)  [below of =11p3]     	{$\overline{O\big(\mathcal{T}_{4,7}\big)}$};	

	\node (10p1)  [below of =5]     	{};
	\node (10p2)  [below of =10p1]     	{};
	\node (10p3)  [below of =10p2]     	{};
	\node (10p4)  [below of =10p3]     	{};
	
	\node[main node] (c45)  [below of =5]     	{$\overline{O\big(\mathcal{T}_{4,5}\big)}$};	
	\node[main node] (c46)  [below of =10p2]     	{$\overline{O\big(\mathcal{T}_{4,6}^{0}\big)}$};
	\node[main node] (c49)  [below of =10p4]     	{$\overline{O\big(\mathcal{T}_{4,8}\big)}$};	
	
	\node (9p1)  [below of =4]     	{};
	\node (9p2)  [below of =9p1]     	{};
	\node (9p3)  [below of =9p2]     	{};

	\node[main node] (c44)  [below of =9p1]     	{$\overline{O\big(\mathcal{T}_{4,4}\big)}$};	
	\node[main node] (c410)  [below of =9p3]     	{$\overline{O\big(\mathcal{T}_{4,9}\big)}$};

	\node (8p1)  [below of =3]     	{};
	\node (8p2)  [below of =8p1]     	{};
	\node (8p3)  [below of =8p2]     	{};

	\node[main node] (c43)  [below of =8p1]     	{$\overline{O\big(\mathcal{T}_{4,3}\big)}$};	
	\node[main node] (c471)  [below of =8p3]     	{$\overline{O\big(\mathcal{T}_{4,6}^{1}\big)}$};
	
	\node (5p1)  [below of =2]     	{};
	\node (5p2)  [below of =5p1]     	{};
	\node (5p3)  [below of =5p2]     	{};

	\node[main node] (c42)  [below of =5p2]     	{$\overline{O\big(\mathcal{T}_{4,2}\big)}$};

	\node (0p1)  [below of =1]     	{};
	\node (0p2)  [below of =0p1]     	{};
	\node (0p3)  [below of =0p2]     	{};

	\node[main node] (c41) [ below  of =0p2]       {$\overline{O\big(\mathcal{T}_{4,1}\big)}$};
	
\path

  (c48) edge   (c46)
  (c48) edge   (c49)

  (c45) edge   (c44)
  (c45) edge   (c471)

  (c46) edge   (c44)

  (c49) edge   (c44)
  (c49) edge   (c43)
  (c49) edge   (c410)

  (c44) edge   (c42)

  (c410) edge  (c42)

  (c43) edge   (c42)

  (c471) edge   (c42)

  (c42) edge   (c41)

  (c47l) edge   (c46)
  (c47l) edge   (c45)
  ;

\end{tikzpicture}}

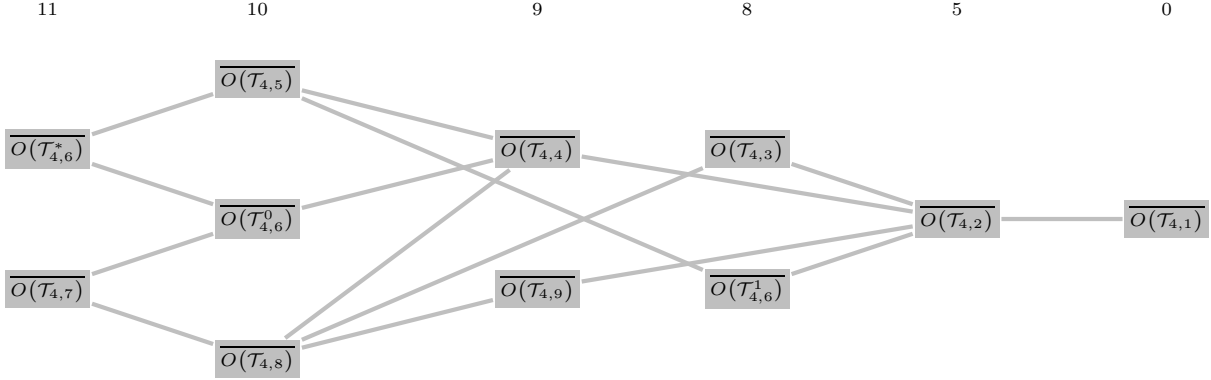
{Figure 2.}  Inclusion graph for orbit closures and dimension. 
\end{center}

\begin{table}[H]
    \centering
    \begin{tabular}{|l|lllll|c|}
			\hline
			${\mathcal  T }$ & Multiplication table &&&&& $\mathrm{dim}\,\mathfrak{Der}$ \ \\
			\hline
\hline

			$\mathcal{T}_{4,1}$ &&&&&& $16$ 	\\
			\hline
			
			$\mathcal{T}_{4,2}$ & $\left[
e_{1},e_{2},e_{1}\right] =e_{3}$&&&&& $9$ 	\\
			\hline
			
			$\mathcal{T}_{4,3}$ & $\left[ e_{1},e_{2},e_{1}\right] =e_{3}$&$\left[
e_{1},e_{2},e_{2}\right] =e_{4}$ &&&& $8$ 	\\
			\hline
			
			$\mathcal{T}_{4,4}$ & $\left[ e_{2},e_{3},e_{2}\right] =e_{4}$&$\left[
e_{3},e_{1},e_{3}\right] =e_{4}$&&&& $7$ \\
			\hline
			
			$\mathcal{T}_{4,5}$ &$\left[ e_{2},e_{3},e_{1}\right] =e_{4}$&$\left[
e_{3},e_{1},e_{2}\right] =e_{4}$&$\left[ e_{2},e_{1},e_{3}\right] =2e_{4}$&$
\left[ e_{2},e_{3},e_{2}\right] =e_{4}$&& $6$ 	\\
			\hline
			
			$\mathcal{T}_{4,6}^{\lambda}$
			&$\left[ e_{1},e_{2},e_{3}\right] =-\left( \lambda +1\right) e_{4}$&$\left[
e_{2},e_{3},e_{1}\right] =\lambda e_{4}$&$\left[ e_{3},e_{1},e_{2}\right]
=e_{4}$&&& $\begin{array}{cc}
8 & \lambda=1,-2,-1/2 \\
6 & \textrm{Otherwise} \end{array}$ 	\\
			\hline
			
			$\mathcal{T}_{4,7}$ &$\left[ e_{1},e_{2},e_{1}\right] =e_{3}$&$\left[
e_{1},e_{2},e_{3}\right] =e_{4}$&$\left[ e_{1},e_{3},e_{2}\right] =e_{4}$&&& $5$ 	\\
			\hline
			
			$\mathcal{T}_{4,8}$ &$\left[ e_{1},e_{2},e_{1}\right] =e_{3}$&$\left[
e_{1},e_{3},e_{1}\right] =e_{4}$&$\left[ e_{1},e_{2},e_{2}\right] =e_{4}$&&& $6$ 	\\
			\hline
			
            $\mathcal{T}_{4,9}$ &$\left[ e_{1},e_{2},e_{1}\right] =e_{3}$&$\left[
e_{1},e_{3},e_{1}\right] =e_{4}$&&&& $7$ 	\\
			\hline

    \end{tabular}
    \caption{The dimension of derivations of   $4$-dimensional nilpotent Lie triple systems}
    \label{tab:clas4}
\end{table}

\begin{table}[H]
    \small
    \centering
    \begin{tabular}{|rcl|llll|}
\hline
\multicolumn{3}{|c|}{\textrm{Degeneration}}  & \multicolumn{4}{|c|}{\textrm{Parametrized basis}} \\
\hline
\hline
$\mathcal{T}_{4,7} $ & $\to$ & $\mathcal{T}_{4,6}^{0}  $ & $
E_{1}(t)= e_1,$ & $
E_{2}(t)= e_2,$ & $
E_{3}(t)= t^{-1} e_3,$ & $
E_{4}(t)= -t^{-1} e_4$\\
\hline
$\mathcal{T}_{4,5} $ & $\to$ & $\mathcal{T}^1_{4,6}  $ & $
E_{1}(t)= e_1,$ & $
E_{2}(t)= te_2,$ & $
E_{3}(t)= e_3,$ & $
E_{4}(t)= te_4$\\
\hline
$\mathcal{T}_{4,8} $ & $\to$ & $\mathcal{T}_{4,3}  $ & $
E_{1}(t)= te_1 ,$ & $
E_{2}(t)= e_2,$ & $
E_{3}(t)= t^2 e_3,$ & $
E_{4}(t)= te_4$\\
\hline
$\mathcal{T}_{4,8} $ & $\to$ & $\mathcal{T}_{4,9}  $ & $
E_{1}(t)= e_1,$ & $
E_{2}(t)= te_2,$ & $
E_{3}(t)= te_3,$ & $
E_{4}(t)= te_4$\\
\hline
$\mathcal{T}_{4,4} $ & $\to$ & $\mathcal{T}_{4,2}  $ & $
E_{1}(t)= e_2,$ & $
E_{2}(t)= t e_3,$ & $
E_{3}(t)= t e_4,$ & $
E_{4}(t)= t e_1$\\
\hline
$\mathcal{T}_{4,9} $ & $\to$ & $\mathcal{T}_{4,2}  $ & $
E_{1}(t)= e_1,$ & $
E_{2}(t)= te_2,$ & $
E_{3}(t)= te_3,$ & $
E_{4}(t)= e_4$\\
\hline
$\mathcal{T}_{4,3} $ & $\to$ & $\mathcal{T}_{4,2}  $ & $
E_{1}(t)= e_1,$ & $
E_{2}(t)= te_2,$ & $
E_{3}(t)= te_3,$ & $
E_{4}(t)= e_4$\\
\hline
$\mathcal{T}_{4,2} $ & $\to$ & $\mathcal{T}_{4,1}  $ & $
E_{1}(t)= t e_1,$ & $
E_{2}(t)= t e_2,$ & $
E_{3}(t)= t e_3,$ & $
E_{4}(t)= t e_4$\\
\hline
$\mathcal{T}_{4,8} $ & $\to$ &  $\mathcal{T}_{4,4} $ & $
E_{1}(t)= t^{-1}e_3,$ & $
E_{2}(t)= -i e_2,$ & $
E_{3}(t)= t e_1,$ & $
E_{4}(t)= t e_4$\\
\hline
$\mathcal{T}_{4,6}^{1} $ & $\to$ &  $\mathcal{T}_{4,2} $ & $
E_{1}(t)= t e_1 - \frac{1}{3 t}e_3,$ & $
E_{2}(t)= e_2,$ & $
E_{3}(t)= e_4,$ & $
E_{4}(t)= e_3$\\
\hline
$\mathcal{T}_{4,7}$ & $\to$ &  $\mathcal{T}_{4,8} $ & $
E_{1}(t)= e_1 - \frac{1}{2t^3}e_2 - \frac{1}{4t^5}e_3,$ & $
E_{2}(t)= \frac{1}{2t}e_2 - \frac{1}{4t^3}e_3,$ & $
E_{3}(t)= \frac{1}{2t}e_3,$ & $
E_{4}(t)= -\frac{1}{4t^4}e_4$\\
\hline
$\mathcal{T}_{4,5}$ & $\to$ &  $\mathcal{T}_{4,4} $ & $
E_{1}(t)= \frac{t}{3}e_1 ,$ & $
E_{2}(t)= e_2,$ & $
E_{3}(t)= -\frac{1}{3t}e_1 + t^{-1}e_2 + e_3,$ & $
E_{4}(t)= e_4$\\
\hline
$\mathcal{T}_{4,6}^{\lambda\neq1}$ &$ \to$&   $\mathcal{T}_{4,4}$ & $
E_{1}(t)= e_2,$ & $
E_{2}(t)= e_1 + \frac{t^{-1}}{(\lambda-1)}e_2,$ & $
E_{3}(t)= -\frac{t^{-1}}{(2\lambda+1)}e_1 - \frac{t^{-2}}{(\lambda^2+\lambda-2)}e_2 + e_3,$ & $
E_{4}(t)= t^{-1}e_4$\\
\hline
    \end{tabular}
    \caption{Degenerations of $4$-dimensional nilpotent Lie triple systems, except of the family $\mathcal{T}_{4,6}^{*}$}
    \label{tab:algdeg}
\end{table}


\begin{table}[H]
    \centering
    \begin{tabular}{|rcl|l|}
\hline
\multicolumn{3}{|c|}{\textrm{Non-degeneration}} & \multicolumn{1}{|c|}{\textrm{Arguments}}\\
\hline
\hline
$\mathcal{T}_{4,7}$ &$\not \to$&  $\mathcal{T}_{4,5}, \mathcal{T}_{4,6}^{\lambda\neq0,-1}$ &
${\mathcal R}= \left\{ \begin{array}{l}
c_{1,2,1}^{3} = - c_{2,1,1}^{3}, c_{1,2,1}^{4} = - c_{2,1,1}^{4}, c_{1,2,2}^{4} = - c_{2,1,2}^{4}, c_{1,2,3}^{4} = - c_{2,1,3}^{4},\\
c_{1,3,1}^{4} = -c_{3,1,1}^4, c_{1,3,2}^{4} = - c_{3,1,2}^{4} = c_{1,2,3}^{4} \textrm{ and }     c_{ijk}^p = 0 \textrm{ otherwise.}  
\end{array} \right\}$\\
\hline
$\mathcal{T}_{4,6}^{\lambda}$ & $\not \to$ &  $\mathcal{T}_{4,6}^{1}$ &
${\mathcal R}= \left\{ \begin{array}{l}
c_{1,2,1}^{4} =  - c_{2,1,1}^{4}, c_{1,2,2}^{4} = - c_{2,1,2}^{4}, c_{1,2,3}^{4} = -c_{2,1,3}^{4} = (1+\lambda) c_{1,3,2}^{4},c_{1,3,1}^{4}= - c_{3,1,1}^{4},\\
c_{1,3,2}^{4}= - c_{3,1,2}^{4}, c_{2,3,1}^{4}= - c_{3,2,1}^{4} = - \lambda c_{1,3,2}^{4} \textrm{ and }  c_{ijk}^p = 0 \textrm{ otherwise.}
 \end{array} \right\}$\\
\hline
$\mathcal{T}_{4,9}$ &$\not \to$&  $\mathcal{T}_{4,3}$ &
${\mathcal R}= \left\{ \begin{array}{l}
c_{1,2,1}^{3}  = - c_{2,1,1}^{3}, c_{1,2,1}^{4} = - c_{2,1,1}^{4}, c_{1,3,1}^{4} = - c_{3,1,1}^{4} \textrm{ and }  c_{ijk}^p = 0 \textrm{ otherwise.} \end{array} \right\}$\\
\hline
$\mathcal{T}_{4,5}$ &$\not \to$&  $\mathcal{T}_{4,9}$, $\mathcal{T}_{4,3}$ &
Corollary \ref{cor:degs} (2) \\
\hline
$\mathcal{T}_{4,6}^{\lambda}$  &$\not \to$&  $\mathcal{T}_{4,9}$, $\mathcal{T}_{4,3}$ &
Corollary \ref{cor:degs} (2)\\
\hline
    \end{tabular}
    \caption{Non-degenerations of $4$-dimensional nilpotent Lie triple systems, except of the family $\mathcal{T}_{4,6}^{*}$}
    \label{tab:algndeg}
\end{table}

\begin{table}[H]
    \centering
    \begin{tabular}{|rcl|llll|l|}
\hline
\multicolumn{3}{|c|}{\textrm{Degeneration}}  & \multicolumn{4}{|c|}{\textrm{Parametrized basis}} & \multicolumn{1}{|c|}{\textrm{Parametrized index}}\\
\hline
\hline
$\mathcal{T}_{4,6}^{*} $ & $\to$ & $\mathcal{T}_{4,5}  $ & $
E_{1}(t)= \frac{1}{2}e_1 + \frac{1}{ 2t} e_2,$ & $
E_{2}(t)= - \frac{1}{2t}e_1 + \frac{1}{2t^2}e_2,$ & $
E_{3}(t)= e_3,$ & $
E_{4}(t)= \frac{1}{2t^2}$ & $f(t) = \frac{2}{1+t}-1$\\
\hline
    \end{tabular}
    \caption{Degenerations of the the $4$-dimensional nilpotent Lie triple systems for the family $\mathcal{T}_{4,6}^{*}$}
    \label{tab:infseriesdeg}
\end{table}

\begin{table}[H]
    \centering
    \begin{tabular}{|rcl|l|}
\hline
\multicolumn{3}{|c|}{\textrm{Non-degeneration}} & \multicolumn{1}{|c|}{\textrm{Arguments}}\\
\hline
\hline
$\mathcal{T}_{4,6}^{*}$ &$\not \to$&  $\mathcal{T}_{4,9}$, $\mathcal{T}_{4,3}$ &
${\mathcal R}= \left\{ \begin{array}{l}
c_{1,2,1}^{4} =  - c_{2,1,1}^{4}, c_{1,2,2}^{4} = - c_{2,1,2}^{4}, c_{1,2,3}^{4} = - c_{2,1,3}^{4},c_{1,3,1}^{4}= - c_{3,1,1}^{4},\\
c_{1,3,2}^{4}= - c_{1,3,2}^{4}, c_{2,3,1}^{4}= -c_{3,2,1}^{4} \textrm{ and }  c_{ijk}^p = 0 \textrm{ otherwise.}
 \end{array} \right\}$\\
\hline
    \end{tabular}
    \caption{Non-degenerations of the $4$-dimensional nilpotent Lie triple systems for the family $\mathcal{T}_{4,6}^{*}$}
    \label{tab:infseriesndeg}
\end{table}

\red 


\black

\end{document}